\documentclass[11pt]{amsart}
\usepackage{pdfsync}
\usepackage[all]{xy}

\usepackage{amsmath,amssymb,amscd,amsfonts,verbatim}
\usepackage{paralist}
\usepackage[mathscr]{eucal}

\usepackage[normalem]{ulem}

\usepackage{enumerate}

\usepackage[pdftex]{hyperref}
\hypersetup{citecolor=blue,linktocpage}

\newtheorem{thm}{Theorem}[section]
\newtheorem{lem}[thm]{Lemma}

\newtheorem{prop}[thm]{Proposition}
\newtheorem{cor}[thm]{Corollary}

\newtheorem{assu-nota}[thm]{Assumption--Notation}

\newtheorem{remark}[thm]{Remark}

\newcommand{\h}{\widehat}

\newcommand{\C}{\mathbb C}

\newcommand{\N}{\mathbb N}
\newcommand{\pp}{\mathbb P}

\DeclareMathOperator{\Pic}{Pic}

\DeclareMathOperator{\Alb}{Alb}

\DeclareMathOperator{\IM}{Im}

\DeclareMathOperator{\Sym}{Sym}
\newcommand{\epsi}{\varepsilon}

\newcommand{\fie}{\varphi}

\newcommand{\OO}{\mathcal{O}}

\usepackage[colorinlistoftodos]{todonotes}

\usepackage{todonotes}

\numberwithin{equation}{section}

\title{Surfaces with canonical map of odd degree}

\author{ Margarida Mendes Lopes} 
\author{ Rita Pardini} 
\author{Roberto Pignatelli}

\begin{document}
\begin{abstract}  
Let $S$ be a smooth complex minimal surface of general type with  $p_g:=h^0(K_S)\ge 4$ whose canonical map is generically finite of odd degree $d>1$ onto a surface $\Sigma$.

We assume  that the general canonical curve of $S$ is smooth and that  $\Sigma$ is ruled by lines, and we prove:
 \begin{itemize}
 \item [--]   $p_g\le d+2$
 \item [--]  $\Sigma$ is a cone over the rational normal curve of degree $p_g-2$ in $\pp^{p_g-1}$
 \item  [--]  $p_g=d+2$ can occur only for $d=3,9,11$.
 \end{itemize}
  As a byproduct, we refine previous results by Beauville and Xiao by proving that if one drops the assumption that $\Sigma$ is ruled by lines  then  $d\le 5$ if $p_g\ge 112$.

The case $d=3$ being completely classified in \cite{triple}, we  focus on $d=5$,  showing that $p_g\le 5$ and that for $p_g=5$ the surface $S$ has a pencil $|C|$ with $C^2=1$ and $K_SC=5$. 

These results suggest that the answer to the question whether  the surfaces with canonical map of odd degree $d>1$ have bounded invariants could be positive, in sharp contrast with the case of even degree.
\medskip


 
 \end{abstract}

\maketitle

 \setcounter{tocdepth}{1}
 
\tableofcontents
 \section{Introduction}
 
 Let $S$ be a smooth minimal surface of general type with $p_g(S):=h^0(K_S)\ge 4$, let $\fie\colon S \to\Sigma \subset \pp^{p_g(S)-1}$ be the canonical map,  and assume that the canonical image $\Sigma$ is a surface. By results of Beauville \cite{beauville} and Xiao \cite{xiao-degree}, for $p_g(S)\gg 0$ the degree of $\fie$  is at most 8 and there are series of examples  with canonical map of degree 8 and unbounded  invariants. 
 However, as noted in \cite[\S~5]{survey}, only a few sporadic  examples where $\fie$ is not birational of odd degree are known, and it is an open question whether such surfaces have bounded invariants.  For degree 3 and $\Sigma$ ruled by lines there is a complete classification under the additional assumption that the general canonical curve of $S$ be smooth (\cite{triple}, cf. also \cite{starnone}): $p_g(S)\le 5$, $\Sigma$ is a cone over the rational normal curve of degree $p_g(S)-2$, the ruling of $\Sigma$ pulls back to a linear pencil $|C|$ on $S$, there are only few possibilities for the triple $(p_g(S),C^2, K_S\cdot C)$,  \underline{and all these possibilities actually occur}.
 
  Here we study the same situation for all  odd values $d>1$ of the degree of the canonical map, still under the assumption that the general canonical curve be smooth. The geometric picture that we obtain is very similar to the   one for $d=3$:
 \begin{itemize}
 \item $p_g(S)\le d+2$, 
 \item $\Sigma$ is a cone over the rational normal curve of degree $p_g(S)-2$
 \item the ruling of $\Sigma$ pulls back to a linear pencil $|C|$ on $S$ and there is a short list of possibilities for the triple $(p_g(S),C^2, K_S\cdot C)$. 
 \end{itemize}
 There is nevertheless a striking  difference:  not only we have not been able to find examples with $d>3$,  but we have ruled out the existence of most of the possible cases. In particular, for $d=5$ we have worked out explicitly the list of possibilities for the triple $(p_g(S),C^2, K_S\cdot C)$,  which is very similar to the one  for $d=3$. There are two ``main'' cases, $p_g=7$, $C^2=1$, $K_S\cdot C=5$ and $p_g=5$, $C^2=2$, $K_S\cdot C=6$, and one expects that  the remaining cases can be obtained as specializations, e.g. by imposing elliptic singularities. However, in contrast with the case $d=3$, we rule out the existence of both main cases and of most of the would-be specializations, so that the only remaining possibilities are $p_g=5$, $C^2=1$, $K_S\cdot C=5$ and $p_g=4$.    

This analysis is the core of the paper and we believe it to be of interest also from the point of view of the techniques: the proofs combine    arguments on the numerical connectedness of  curves on (possibly singular) surfaces, classification results for  surfaces of small degree in projective spaces, the structure of the gap set of Weierstra\ss\  points on curves of small genus and the fact that the parity of theta characteristics is preserved in families. Actually, to exclude  the case $p_g=5$, $C^2=2$, $K_S\cdot C$ (the hardest!) we make use of an infinitesimal version of the preservation of the parity of theta characteristics \cite{InfinitesimalParityOfTheta}. We refer the reader to \S \ref{sec: results} for a complete description of the results and an outline of the strategy of proof; we just mention here Corollary \ref{cor: d<7}, that implies that surfaces  with canonical map of odd degree $\ge 7$ and smooth general canonical curve have bounded degree, improving results by Beauville \cite{beauville} and Xiao \cite{xiao-degree}.
Summing up, the results contained in this paper seem to point to a positive answer to the question whether surfaces with non birational  canonical map of odd degree have bounded invariants.

Finally, a comment on the assumption, made throughout all the paper,  that the general canonical curve be smooth.
 On one hand it can be  regarded as a mild restriction, since it is ``almost always'' satisfied in concrete examples, but on the other hand it would be very desirable to remove it and have unconditional statements. Unfortunately, this assumption is essential at all crucial points of our proofs, hence at the moment we have no clue on how to dispense with it.

\subsection{Acknowledgements} Research partially supported by Funda\c c\~ao para a Ci\^encia e Tecnologia (FCT), Portugal, through CAMGSD, IST-ID, projects UID/4459/2025 (DOI: 10.54499/UID/04459/2025), UIDB/04459/2020 and  UIDP/04459/2020, by the European Union - Next Generation EU, Mission 4 Component 2 - CUP E53D23005400001 and by  PRIN 2022BTA242 ``Geometry of algebraic structures: Moduli, Invariants, Deformation'' of Italian MUR.

Part of this project was carried out while the authors were guests of the Research in Pairs program of CIRM-Trento.

The first named author is a member of Centro de An\'alise Matem\'atica, Geometria e Sistemas Din\^amicos.
 The second and third named authors are  members of GNSAGA of INDAM.

\subsection{Notation and conventions} 
We work over $\C$. Varieties embedded in a projective space are assumed to be non degenerate, unless otherwise stated.
A {\em surface} $X$ is a normal  complex projective surface; when $X$ is smooth or has canonical singularities we denote as usual by $K_X$ the canonical class and write $p_g(X):=h^0(K_X)$ (\!{\em geometric genus})  and $q(X):=h^1(\OO_X)$ (\!{\em irregularity}). 

A   {\it curve} $D$ on a smooth surface $S$ is  a non-zero effective divisor (in particular $D$ is Gorenstein).

Given a curve $D$, we denote  its dualizing sheaf  by $\omega_D$  and  its arithmetic genus by $p_a(D)$,  where by definition $1-p_a(D)=h^0(D,\OO_D)-h^1(D,\OO_D)$.
 A curve  $D$ is {\em $m$-connected} if for every decomposition $D=A+B$ as the sum of two curves $A$ and $B$ one has $A B \geq m$.
 A {\em $(-2)$-curve} (or {\em nodal curve}) on a smooth surface $S$ is an irreducible curve that is isomorphic to $\pp^1$ and has self-intersection $(-2)$.

A pencil on a smooth surface $S$ is a rational map   with connected fibers $C\to B$, where $B$ is a smooth curve. The pencil is said to be {\em linear} if $B=\pp^1$, and {\em irrational of genus $g$} otherwise, where $g$ is the genus of $B$.

\section{Preliminaries}
\subsection{Surfaces ruled by lines}

 Let $\Sigma\subset \pp^n$ be a non-degenerate surface and let  $\rho\colon Y\to \Sigma$ be the minimal resolution of singularities. We say that $\Sigma$  is   {\em linearly normal } if the linear system $\rho^*|\OO_{\Sigma}(1)|$ is complete.
 
 We recall the following result from \cite{cdm}:
 \begin{thm}\label{thm: cdm}
Let $\Sigma\subset \pp^n$ be a  surface ruled by lines. If $\Sigma$ is linearly normal then:
\begin{enumerate}
\item if $q(\Sigma)=0$, then $\deg\Sigma=n-1$ and $\Sigma$ is either a cone or a rational normal scroll;
\item if $q(\Sigma)>0$, then $\Sigma$ is a scroll with $n-1+q(\Sigma)\le \deg\Sigma\le n-1+2q(\Sigma)$ and if $ \deg\Sigma= n-1+q(\Sigma)$ then $\Sigma$ is a cone. 
\end{enumerate}
\end{thm}

We will also use the following result  obtained independently by  Xiao Gang and Miles Reid:

\begin{thm}{\cite[Lemma 1]{xiao-degree}, \cite{miles-small}}\label{thm: xiao-ruled}\
Let $Y\subset \pp^n$ be a non-degenerate surface such that $\deg Y < \frac 43(n-2)$.  

Then either $Y$ is ruled by lines or $n=9$ and $Y\subset \pp^9$ is the image of $\pp^2$ via the complete system $|\OO_{\pp^2}(3)|$.
\end{thm}
\subsection{Weierstra\ss\ points on Gorenstein curves}
Let $C$ be a Gorenstein curve with $h^0({\mathcal O}_C)=1$, and let $p$ be a smooth point of $C$. Let $g$ denote the arithmetic genus of $C$.

As for smooth curves, the Riemann-Roch Theorem and Serre Duality imply that $h^0(C,np)=n-g+1$ for $n\geq 2g-1$. Moreover $h^0(C,np)-h^0(C,(n-1)p)=0$ or $1$, so we can borrow the usual terminology and say that $n$ is a gap for $p$ if $h^0(C,np)=h^0(C,(n-1)p)$, and that $p$ is a Weierstra\ss\ point if its set of gaps is not $\{1,\ldots, g-1\}$. Observe that as in the smooth case the cardinality of the set of gaps is $g$.

\begin{lem}\label{lem: semigroup}
Let $C$ be a Gorenstein curve with $h^0({\mathcal O}_C)=1$, and let $p$ be a smooth point of $C$. Then the complement of the set of gaps of $p$ is a semigroup.
\end{lem}
\begin{proof}
 Let  $C_p$ be the irreducible component of $C$ that contains  $p$. Observe that the restriction map $H^0(C,np) \rightarrow H^0(C_p,np)$ is injective: a nonzero section in the kernel would be a nonzero section of $\OO_C$ that vanishes on $C_p$, contradicting the assumption that $h^0({\mathcal O}_C)=1$. So if  $s_i\in H^0(C,n_ip)$, $i=1,2$, are nonzero sections then their restrictions on $C_p$ is nonzero; as $C_p$ is irreducible then $s_1s_2$ is also nonzero on $C_p$ and then on $C$ too. The statement follows from the same argument as in the case of smooth curves. 
\end{proof}

If $\omega_C$ is a multiple of $p$ we have the following easy lemma.

\begin{lem}\label{lemma: gaps come in pairs}
Let $C$ be a Gorenstein curve of genus $g\ge 2$ with $h^0({\mathcal O}_C)=1$, and let $p$ be a smooth point of $C$.
Assume $\omega_C \cong {\mathcal O}_C((2g-2)p)$.

Then $p$ is a Weierstra\ss\ point of $C$ with minimal gap $1$ and maximal gap $2g-1$.
The remaining $g-2$ gaps are one for each pair $\{2,2g-3\}$, $\{3,2g-4\}$, $\{4, 2g-5\}$, $\ldots$, $\{g-1,g\}$.
\end{lem}
\begin{proof}
Since $\omega_C={\mathcal O}_C((2g-2)p)$ then $2g-1$ is a gap, the maximal possible gap. 
The Riemann-Roch Theorem and Serre Duality imply that there is exactly one gap on each of the $g-1$ pairs $\{1,2g-2\}$, $\{2, 2g-3\}$, $\ldots$, $\{g-1,g\}$.
 In fact $1$ is obviously a gap, i.e. the gap of the pair  $\{1,2g-2\}$. 
\end{proof}

\subsection{Parity of theta characteristics}

 We recall  a  generalized version  due to Harris (\cite[Theorem 1.10.(i)]{Har}, see also \cite{Cor89})  of the classical fact (cf. \cite{Mum71}) that the parity of theta characteristic is constant in families. 

\begin{thm}{\cite[Theorem 1.10.(i)]{Har}}\label{thm: parity of theta}\
 Let $\Delta$ be an irreducible  variety,  let $\pi \colon X\rightarrow \Delta$ be a proper flat map with fibers $C_t:=\pi^{-1}(t)$ reduced curves,  let   $\mathcal L$ be  a line bundle  on $X$  and set $\mathcal L_t:={\mathcal L}|_{C_t}$.
 
 If  ${\mathcal L}_t^{\otimes 2}
\cong \omega_{C_t}$ for all  $t\in \Delta$,  then the function $t\mapsto h^0(C_t,{\mathcal L}_t)$ is constant modulo 2.
\end{thm}

We also need  the infinitesimal version of Theorem \ref{thm: parity of theta} from \cite[Corollary 1.2]{InfinitesimalParityOfTheta}
\begin{thm}\label{thm: infinitesimal}
In the assumptions of Theorem \ref{thm: parity of theta}, suppose in addition that $\Delta$ is a smooth curve.

Then
\begin{enumerate}
\item for all  $t\in \Delta$, for all $k \in {\mathbb N}$,  $h^0(kC_t,{\mathcal L}|_{kC_t})$ equals $kh^0(C_t,{\mathcal L}|_{C_t})$ modulo $2$;
\item  there is a coherent sheaf ${\mathcal T}$ on $\Delta$ such that the torsion subsheaf of $R^1\pi_*{\mathcal L}$ is isomorphic to ${\mathcal T} \oplus {\mathcal T}$.
\end{enumerate}
\end{thm}

\section{Statement of the  results}\label{sec: results}

 The next theorem summarizes our main results.
 
 \begin{thm}\label{thm: 42} Let $S$ be a minimal surface of general type with $p_g(S)\ge 4$ such that the canonical map $\fie\colon S\to \pp^{p_g(S)-1}$ is generically finite. Assume that:
\begin{itemize}
\item[(a)]  the general canonical curve $D\in |K_S|$ is smooth;
\item[(b)]  the degree $d$ of $\fie$ is  odd;
\item[(c)] the canonical image $\Sigma$ is ruled by lines.
\end{itemize}
Then:
\begin{enumerate}
\item $p_g(S)\le d+2$ and $\Sigma$ is the cone over the rational normal curve of degree $p_g(S)-2$;
\item the strict transforms of the ruling of $\Sigma$ induce a linear pencil of non hyperelliptic curves $|C|$ with $C^2>0$ and $g(C)\ge \frac{d+3}{2}$; 
\item $\fie$ separates the curves of $|C|$;
\item if $p_g(S)=d+2$, then $d=3, 9,11$, the system $|K_S|$ is base point free, $C^2=1$ and $K_S=dC$;
\item if $d=5$, then $p_g(S)\le 5$; if $p_g(S)=5$ then  $C^2=1$, $K_SC=5$.
\end{enumerate} 
\end{thm}
 The proof takes up the rest of the paper:  proving  (iv) and (v)  is by far the hardest part. 

In the next section we prove Proposition \ref{prop: max-case}, which is a weaker version of Theorem \ref{thm: 42}. More precisely Proposition \ref{prop: max-case} proves (i), (ii), (iii) and a weaker version of (iv), namely that if $p_g(S)=d+2$, then $d \le 11$, the system $|K_S|$ is base point free, $C^2=1$ and $K_S=dC$. The proof of (iv) is completed by excluding that $p_g(S)=d+2$  occur for $d=5$ or $7$: this is Proposition \ref{prop: no d=5,7}. As a consequence $d=5$ implies $p_g \le 6$: (v) then follows  by Propositions \ref{prop: cases5}, \ref{prop: no pg=6} and \ref{prop: no C^2=2}.

The conclusion of Theorem \ref{thm: 42} holds also if one  assumes that $\Sigma$ is a surface of minimal degree $p_g-2$ in $\pp^{p_g-1}$. In fact by Theorem \ref{thm: xiao-ruled} the only surface of minimal degree which is not ruled by lines is the Veronese surface  $V_{2,2}\subset \pp^5$ and we have the following:
\begin{prop}
 Let $S$ be a minimal surface of general type  such that the general canonical curve $D\in |K_S|$ is smooth. If the  canonical image of $S$   is the Veronese surface  $V_{2,2}\subset \pp^5$ then the degree $d$ of the canonical map $\fie$ of $S$ is even.
\end{prop}
\begin{proof}
Let $C$ be the pull back via $\fie$ of a  line in $\pp^2\cong V_{2,2}$, so that $K_S=2C$.
 If  $\fie$ is  not a morphism, then $|C|$ has at least a base point $x\in S$. Since $|K_S|=\fie^*|\OO_{\pp^2}(2)|$ and the map $\Sym^2H^0(\OO_{\pp^2}(1))\to H^0(\OO_{\pp^2}(2))$ is surjective, it follows that $x$ is a base point of $|K_S|$ of multiplicity at least 2, namely the general canonical curve is singular at $x$, contradicting the assumptions. So $\fie$ is a morphism and $C^2=d$. We conclude by observing that  $C^2=C(K_S-C)$ is even by the adjunction formula.
 \end{proof}

Theorem \ref{thm: 42} has the following almost immediate consequence:

\begin{cor}\label{cor: d<7} Let $S$ be a minimal surface of general type such that the canonical map $\fie\colon S\to \pp^{p_g(S)-1}$ is generically finite. Assume that:
\begin{itemize}
\item[(a)]  the general canonical curve $D\in |K_S|$ is smooth;
\item[(b)] the degree $d$ of $\fie$ is odd;
\end{itemize}
Then:
\begin{itemize}
\item if  $d=7$ then $p_g \leq 111$;
\item if  $d=9$ then $p_g \leq 15$;
\item if  $d=11$ then $p_g \leq 13$;
\item if  $d \ge 13$ then $p_g \leq 2+\frac{27}{d-9}$.
\end{itemize}
\end{cor}
\begin{proof} 
As in Theorem \ref{thm: 42}, we denote by $\Sigma$ the canonical image $\varphi(S)$.
By the Bogomolov-Miyaoka-Yau inequality 
\begin{equation}\label{eqn: Bogomolov-Miyaoka-Yau}
d\deg\Sigma\le K^2_S\le 9(p_g+1)
\end{equation}
and $\deg \Sigma \geq p_g-2$ we obtain that 
\begin{equation}\label{eqn: inequality among pg and d}
p_g  \leq \frac{2d+9}{d-9}=2+\frac{27}{d-9} \text{ when } d\geq 10.
\end{equation}

If $d=11$, \eqref{eqn: Bogomolov-Miyaoka-Yau} implies $\deg \Sigma \le \frac{9}{11} (p_g+1)$, which is, for $p_g \ge 10$, smaller than $\frac 43(p_g-3)$: then by Theorem \ref{thm: xiao-ruled}, the surface $\Sigma$ is ruled by lines and  Theorem \ref{thm: 42} gives $p_g\le d+2=13$.

If $d=9$, \eqref{eqn: Bogomolov-Miyaoka-Yau} implies $\deg \Sigma \le  p_g+1$, which is, for $p_g \ge 16$, smaller than $\frac 43(p_g-3)$, forcing $\Sigma$ to be ruled the lines: then Theorem \ref{thm: 42} gives $p_g\le 11$, a contradiction. The same  argument proves that $d=7$ implies $p_g \le 111$. 
\end{proof}

\begin{remark} Xiao Gang (\cite{xiao-degree}) proved that if the canonical map of a surface has degree $9$ then $p_g \leq 132$; assuming that the general canonical curve  be smooth we obtain a much stronger bound. Moreover our assumption allows us to get a bound also for $d=7$.
\end{remark}

\section{General properties}

Here we prove a partial version of Theorem \ref{thm: 42}. 

More precisely, note  first that   Proposition \ref{prop: max-case} below has exactly the same assumptions as  Theorem \ref{thm: 42}. Moreover the claims  (i),(ii),(iii) of Proposition  \ref{prop: max-case} are the same as in Theorem \ref{thm: 42}, whereas claim (iv) of Proposition  \ref{prop: max-case} is weaker than claim (iv) of Theorem \ref{thm: 42}. 

\begin{prop}\label{prop: max-case}  
Let $S$ be a minimal surface of general type with $p_g(S)\ge 4$ such that the canonical map $\fie\colon S\to \pp^{p_g(S)-1}$ is generically finite. Assume that:
\begin{itemize}
\item[(a)]  the general canonical curve $D\in |K_S|$ is smooth;
\item[(b)]  the degree $d$ of $\fie$ is  odd;
\item[(c)] the canonical image $\Sigma$ is ruled by lines.
\end{itemize}
Then:
\begin{enumerate}
\item $p_g(S)\le d+2$ and $\Sigma$ is the cone over the rational normal curve of degree $p_g(S)-2$;
\item the strict transforms of the rulings of $\Sigma$ induce a linear pencil of non-hyperelliptic curves $|C|$ with $C^2>0$ and $g(C)\ge \frac{d+3}{2}$; 
\item $\fie$ separates the curves of $|C|$;
\item if $p_g(S)=d+2$, then $d \le 11$, the system $|K_S|$ is base point free, $C^2=1$ and $K_S=dC$.
\end{enumerate} 
\end{prop}

\begin{proof} Set $p_g(S)=:n+1$.
  Let $\epsi\colon \h S \to S$  be the blow up at the, possibly infinitely near, base points of $|K_S|$, and denote by $E_1,\dots E_k$ the corresponding, possibly reducible, $(-1)$-curves. We denote by $\h\fie\colon \h S\to \pp^{n}$   the morphism induced by $\fie$.
Since the general canonical curve is smooth by assumption,  the moving part of $|K_{\h S}|$ is $M:= K_{\h S}-2\sum E_i$. 

Let $h\colon \h S\to B$ be the pencil (with connected fibers) induced by the ruling  of the canonical image and denote by $\h C$ a general fiber of $h$.
The strict transform $G$ of a general  line of $\Sigma$ is numerically equivalent to $\delta\h C$, where $\delta$ is a divisor of $d$, 
 so $G^2=\delta^2\h C^2\ge 0$. In addition $MG=d$, so  $K_{\h S}G=d+2\sum E_iG$ is odd and the adjunction formula gives $G^2>0$, hence also $\h C^2>0$. So $h$ is not a morphism, $B=\pp^1$ and  $\Sigma$ is a cone over a rational curve. 
 
  By Theorem \ref{thm: cdm},  $\Sigma$ has degree  $n-1$ and  it  is the cone over the rational normal curve of degree $n-1$.
The Hodge index theorem gives:
$$ G^2(d(p_g(S)-2)) =G^2M^2\le (MG)^2=d^2$$
from which 
\begin{equation}\label{eq: pg<d-2}
p_g(S) \le 2+\frac{d}{G^2}\le d+2.
\end{equation}
We have then proved (i),

\smallskip

Assume that $\h \fie$ does not separate the curves of $|\h C|$, namely that $\delta>1$.  The assumption $p_g \geq 4$ and \eqref{eqn: Bogomolov-Miyaoka-Yau} imply $d \leq 22$. Since $G^2 \geq \delta^2$, then \eqref{eq: pg<d-2} implies $\delta <5$. 
Since $\delta$ is odd, we obtain $\delta=3$.
Then  \eqref{eq: pg<d-2} gives $d \geq 18$. On the other hand $d$ is odd, divisible by $\delta=3$ and not greater than $22$, so $d=21$.  This implies, using  \eqref{eq: pg<d-2} again, $p_g=4$. Then pulling back a hyperplane section through the vertex of the quadric cone $\Sigma$ we find that $6{\h C} \leq K_{\h S}$, which implies $p_g \geq 7$, a contradiction. This shows that $\fie$ separates the curves of $|C|$, proving (iii).

\smallskip

Then $G=\h C$ and $2g ({\h C})-2={\h C}^2+K_{\h S}{\h C}\ge 1+K_{\h S}{\h C}\ge 1+M{\h C}=1+d$, so $g(C) \ge g(\h C)\ge \frac{d+3}2$, completing the proof of (ii).

\smallskip

Now we assume $p_g(S)=d+2$. The Bogomolov-Miyaoka-Yau inequality \eqref{eqn: Bogomolov-Miyaoka-Yau} gives $d\le 11$. 
The Index Theorem gives:
$$d^2=d\left(p_g (S)-2 \right)=d\left( \deg \Sigma \right) = M^2\le {\h C}^2M^2\le ({\h C}M)^2=d^2,$$
so ${\h C}^2=1$ and $M\sim_{num}d{\h C}$. On the other hand  $\Sigma$ is the cone over the  rational normal curve of degree $d$, so  $M= d{\h C}+Z$, with $Z\ge 0$. Comparing the two expressions for $M$ we get $Z\sim_{num}0$ and therefore $Z=0$ and $M=d{\h C}$. Now assume by contradiction  that $|K_S|$ has base points, namely  ${\h S}$ contains a $(-1)$-curve $E$;  since we assumed the general canonical curve to be smooth,  all the base points of $|K_S|$ are simple,  hence there is a
$(-1)$-curve $E$ such that $ME = 1$, contradicting the fact that $M$ is divisible by $d$ in $\Pic(S)$. So $\h S=S$ and  $dC= M=K_S$.
\end{proof}

After the first version of this paper appeared online, it was brought to our attention that \cite[Theorem 3.2]{GP} had observed that a surface whose canonical map has odd degree and whose canonical system is base-point free (a stronger assumption than our hypothesis that the general canonical curve is smooth) cannot have as its canonical image a smooth surface ruled by lines.

\section{Some more results on the case \texorpdfstring{$d=5$}{d=5}}

To study the cases of Theorem \ref{thm: 42} with $p_g \leq d+1$ we analyze  some special canonical divisors on $S$, borrowing some arguments from \cite{triple}.  

\subsection{A useful canonical divisor}
In the situation of Proposition \ref{prop: max-case}, we consider a general point $z$ in the cone $\Sigma \subset {\mathbb P}^{p_g(S)-1}$ and  enumerate as $x_1,\ldots, x_d$ the points of $\varphi^{-1}(z)$. If $H$ is a general hyperplane through $z$, then 
$H_{|\Sigma}$ is a smooth rational curve whose pull-back on $S$ is a smooth canonical curve $D$. Then  $h^0(D,x_1+\ldots +x_d)\geq h^0({\mathcal O}_{{\mathbb P}^1}(1))=2$. By Serre duality,  $x_1, \ldots , x_d$ do not impose independent conditions on the canonical system of $D$ and therefore, a fortiori, on the linear system $|2K_S|$ on $S$.

Consider the pencil $|C|$  on $S$ induced by  the ruling of $\Sigma$. Since $\Sigma$ is a cone, for any general point $z$ of $\Sigma$ there is a canonical curve $D_z$ on $S$ such that $D_z=2C_z+B$, where $C_z$ is the fibre of $|C|$ corresponding to the line through $z$, $B\geq 0$ and all points $x_1,\ldots, x_d$ belong to $C_z$ and not to $B$. 

In this situation we can use \cite[Lemma 2.3]{triple}, that we recall here.
\begin{lem}\label{lem: characterization of dependence}  
Let $S$ be a smooth surface, let $D$ be a curve on $S$ and let $x_1, \ldots , x_d$ be distinct singular points of $D$. 
Let $p\colon \hat S \to S$ be the blow-up at $x_1,\ldots ,x_d$ and  let $ E_1, \ldots , E_d$ be the corresponding exceptional curves. 
Setting  $D^\prime=p^*D-E_1-\ldots -E_d$ and  $D^\prime{}^\prime=p^*D-2E_1-\ldots -2E_d$, the following two conditions are equivalent:

\begin{enumerate}
\item the points $x_1, \ldots , x_d$ do not impose independent conditions on the linear system $|K_S + D|$;
\item the restriction map $H^0(D^\prime,{\OO}_{D^\prime})\to
H^0(D^\prime{}^\prime,{\OO}_{D^\prime{}^\prime})$ is not surjective.
\end{enumerate}
\end{lem}

To apply the previous Lemma, we will also need the following result, which  appears implicitly in \cite{bom}.
\begin{lem}\label{lem: characterization decomposition}  
Let $S$ be a smooth surface, let $D$ be a  1-connected curve on $S$ and let $x_1, \ldots , x_d$ be distinct singular points of $D$. 
Let $p\colon \hat S \to S$ be the blow-up at $x_1,\ldots ,x_d$ and  let $ E_1, \ldots , E_d$ be the corresponding exceptional curves. 
Setting  $D^\prime=p^*D-E_1-\ldots -E_d$ and  $D^\prime{}^\prime=p^*D-2E_1-\ldots -2E_d$, then  

\begin{enumerate}
\item $D^\prime$ is 1-connected; 
\item if $D^\prime{}^\prime$ is not 1-connected, then $D$ has a decomposition as $D=D_1+D_2$ where  $ D_1D_2=m\leq d$  and at least $m$ of the points $x_1,..., x_d$ belong to both $D_1$ and $D_2$. 
\end{enumerate}
\end{lem}
\begin{proof} Consider any decomposition of $D^\prime$ or $D^\prime{}^\prime$ as a sum of two curves $B_1$ and $B_2$ and write $\Delta_i= B_i+(B_iE_1)E_1+\dots +(B_iE_d)E_d$, for $i=1,2$.  Since $\Delta_i E_j=0$ for $i=1,2$ and $j=1,\dots d$, there are  divisors $D_1,D_2\ge 0$ such that $D=D_1+D_2$ and  $p^*D_i=\Delta_i$, for $i=1,2$.   

In the case of $D^\prime$,  $E_j(B_1+B_2)=1$, whilst in the case of $D^\prime{}^\prime$,  $E_j(B_1+B_2)=2$, for every $j=1,\dots d$.  So in the case of $D'$, $(E_jB_1)(E_jB_2)\leq 0$,  for every $j=1,\dots d$. 

In the case of $D^\prime{}^\prime$,  $(E_jB_1)(E_jB_2)>0 $, if and only if $E_jB_1=E_jB_2=1$ (and if so  $(E_jB_1)(E_jB_2)=1 $). 

If, say, $D_1=0$, then  every component of $B_1$ is one of the $E_j$, and so intersects $B_1$ negatively,  implying that $B_1B_2>0$, because $E_j(B_1+B_2)>0$.

If  $D_1,D_2>0$ then  $\Delta_1\Delta_2=D_1D_2$. 

We have  $$B_1B_2= \Delta_1\Delta_2- (E_1B_1)(E_1B_2)-\dots-(E_jB_1)(E_jB_2)=$$ $$=D_1D_2- (E_1B_1)(E_1B_2)-\dots-(E_jB_1)(E_jB_2).$$
In the case of $D'$, because $ (E_jB_1)(E_jB_2)\leq 0$, for $j=1,\dots,d$ and $D$ is 1-connected, we have necessarily $B_1B_2>0$. Hence $D^\prime$ is 1-connected. 

In the case of $D^\prime{}^\prime$,  if $D_1D_2=m$, $B_1B_2\leq 0$ can only occur if   at least $m$ of the $E_j$ satisfy  $(E_jB_1)(E_jB_2)=1 $ (i.e, at least $m$ of the  points $P_1,\dots,P_d$ belong to $D_1\cap D_2$).
\end{proof}

We apply the above to $D_z$. Since $D_z$ is 1-connected,  $D^\prime$ is also 1-connected by the previous Lemma \ref{lem: characterization decomposition} and so  $h^0(D^\prime,{\OO}_{D^\prime})=1$.   Since $x_1,...,x_d$ do not impose independent conditions to $|2K_S|$,  by Lemma \ref{lem: characterization of dependence}  
$h^0(D^\prime{}^\prime,{\OO}_{D^\prime{}^\prime})\ge 2$.   
Note that, by our choice of $D_z$, none of the $E_i$ is a component of $D^\prime{}^\prime$.  Applying Lemma A.1 of \cite{cfm}  to the sheaf ${\OO}_{D^\prime{}^\prime}$, we obtain a decomposition $D^\prime{}^\prime= A_1+A_2$ with $A_1,A_2$ effective non-zero divisors such that  $A_1A_2\leq 0$ and every component $\theta$ of $A_1$ satisfies $\theta A_2\leq 0$. 

By Lemma \ref{lem: characterization decomposition},   this decomposition of $D^\prime{}^\prime$ yields in turn a decomposition of $D_z$  as $D_z=D_1+D_2$ where  $ D_1D_2=m\leq d$  and at least $m$ of the points $x_1,..., x_d$ belong to both $D_1$ and $D_2$. 

Note that $m\geq 2$, by the 2-connectedness of the canonical curves.  Furthermore, any component $\theta$ of $D_1$ not passing through the points $x_1,...,x_d$ satisfies $\theta D_2\leq 0$. 
  By our choice of $D_z$ we conclude that $C_z$ is a common component of $D_1$ and $D_2$ and that every other component  $\theta$ (if any) of $D_1$ satisfies $\theta D_2\leq 0$.  Furthermore, since $K_SC\geq d$, $K_SD_i\geq d$ for $i=1,2$.  

Since $D_z$ is a canonical curve, $m$ is even.  If $d$ is odd, $m< d$ and  so $D_i^2+m=K_SD_i\geq d$  implies that $D_i^2\ge d-m>0$, for $i=1,2$.

In conclusion, we have 
\begin{prop}\label{prop: coming back}
In the assumptions of  Theorem \ref{thm: 42}, let $z$ be a general point of $\Sigma$ and let $x_1, \ldots, x_d$ be the preimages of $z$ via the canonical map. 

There  is a canonical divisor $D$ on $S$ with a decomposition $D=D_1+D_2$
such that the even number $m= D_1D_2$ is strictly smaller than $d$, and at least $m$ of the points  $x_j$ are simple points of $D_1$; furthermore any component  $\theta$ of $D_1$ not passing through the points $x_1, \ldots , x_d$ satisfies  $\theta D_2\leq 0$. 

Let $C_z$ be the element of the pencil $|C|$ on $S$ induced by the ruling of $\Sigma$ containing the $x_j$. Then
\begin{enumerate}

\item $C_z\leq D_1 \leq D_2$;
\item $1 \le C^2 \le D_1^2 \le D_2^2$;
\item if $D_1-C_z>0$, then $D_1^2\geq C^2+2$.
\end{enumerate}
\end{prop}
\begin{proof}
From the explanations above we only have to prove (ii), (iii) and  that $D_1\leq D_2$.

  Set $G$ for the greatest common divisor of $D_1$ and $D_2$, and write $G_j=D_j-G$, for $j=1,2$.  We know that all components $\theta$ of $D_1$ different from $C_z$ do satisfy   $\theta D_2\leq 0$, and so  $G_1D_2\leq 0$.  Since $G_1$ and $G_2$  have no common components,  $G_1G_2\ge 0$.  So it follows $G_1(K_S-G_1)=G_1(2D_2-G_2)\le 0$. The $2-$connectedness of every canonical divisor implies  $G_1=0$. Hence $D_1\leq D_2$.

Then, by the nefness of $K_S$, $D_2^2=D_1^2+2D_1G_2+G_2^2=D_1^2+K_SG_2\ge D_1^2$.

Recall that $C^2 \ge 1$ by Proposition \ref{prop: max-case}, (ii). If $D_1=C_z$, (ii) is proved.

Else,  $D_1\neq C_z$. We write $D_1=C_z+P$, with $P\neq 0$. Since the only component of $D_1$ through the points $x_j$ is $C_z$ then  $PD_2\leq 0$.
 Now  $P(K_S-P)= PC_z+PD_2\le PC_z$.  
 So by the 2-connectedness  of canonical divisors we have $PC_z\geq 2$. 
 
 Since  $K_S$ is nef  we have $PD_1 =PK_S-PD_2 \geq 0$. 
On the other hand
$D_1^2=C_zD_1+PD_1=C_z^2+PC_z+PD_1$ and so $D_1^2\geq C_z^2+2$.\end{proof}

\subsection{Application to the case $d=5$}
The previous results can be used to determine the numerical possibilities for $C^2$, $K_SC$ and $p_g(S)$  in the situation of Theorem \ref{thm: 42} for every value of $d$. Below we do it for $d=5$, since this is the only case we analyze in greater detail in the paper. 

\begin{prop}\label{prop: cases5}
In the assumptions of  Theorem \ref{thm: 42}, let $|C|$ be the pencil on $S$ induced by  the ruling of $\Sigma$.

Assume $d=5$. Then
\begin{enumerate}
\item if $p_g \ge 6$ then $C^2=1$, $K_SC=5$;
\item if $p_g=5$ then either $C^2=1$, $K_SC=5$ or $C^2=2$, $K_SC=6$. In the latter case $K_S=3C$.
\end{enumerate}
\end{prop}
\begin{proof}
We consider the canonical divisor $D=D_1+D_2$ from Proposition \ref{prop: coming back}. Then  $D_1D_2$ is either $2$ or $4$. In addition $C_z\le D_1$, so $K_SD_1\ge K_SC_z\ge 5$. 
If $D_1D_2=2$,  then
$5 \le K_SD_1=D_1^2+D_1D_2$ gives $D^2_1 \ge 3$, contradicting $D_1^2 \le D_2^2$ (Proposition \ref{prop: coming back}) and the index theorem.
So $D_1D_2=4$.

Then $5(p_g-2) \le K_S^2 =D_1^2+D_2^2+8$, so  $D_1^2+D_2^2 \ge 5p_g-18$. 
By the index theorem $D_1^2 D_2^2 \le 16$.

If $p_g \ge 6$, then $ D_1^2+D_2^2\ge 12$. Since $D_1^2 D_2^2 \le 16$, the only possibility is $D_1^2=1$
and then $D_1=C_z$ by Proposition \ref{prop: coming back}. It follows $C^2=1$, $K_SC=(D_1+D_2)D_1=5$. We have proved (i).

If $p_g=5$, then $\Sigma$ has degree $3$  and by construction $D=2C_z+C_0+Z$ where $C_0$ is an element of $|C|$ and $Z \ge 0$ is contracted by $\varphi$ to the vertex of $\Sigma$. Since $D_1 \le D_2$ then $D_2 \geq C_z+C_0$. Since $K_S$ is nef
\[
D_2^2 =K_SD_2-D_1D_2=K_SD_2-4\ge 2K_SC-4\ge 10-4=6.
\]
By the index theorem it follows $D_1^2 \le \left\lfloor \frac{16}6 \right\rfloor=2$ and then $D_1=C_z$ by Proposition \ref{prop: coming back}. So either $C^2=1$, $K_SC=5$ or $C^2=2$, $K_SC=6$.

Finally, if $C^2=2$ then $K_SD_2\geq 2K_SC= 12$  implies $D_2^2 \ge 8$, and then    by the index theorem $D_2$ is numerically equivalent to $2C$, and then $K_S$ is numerically equivalent to $3C$. So the effective divisor  $Z$ is numerically equivalent to $0$ and so $Z=0$, which concludes the proof.
\end{proof}

\section{The linear systems  \texorpdfstring{$|mC|$}{|mC|}}
\label{sec: exclusion of d=7}

In this section we collect most of the technical results needed to prove  points  (iv) and (v) of Theorem \ref{thm: 42}.

Throughout all the section we will make the following:
\begin{assu-nota}\label{assu: C21}
 $S$ is a minimal surface of general type with  an irreducible  pencil  $|C|$ such that  $C^2=1$. In these assumptions the base locus  of  $|C|$  consists of  a simple base point that we denote by $p$.
 \end{assu-nota}
 \begin{lem}\label{lem: h^0}
Let $m>0 $ be an integer. 
Then:
\begin{enumerate}
\item $h^0(mC)\ge m+1$;
\item if $h^0(mC)\ge m+2$, then $h^0((m+h)C)\ge m+2h+2$ for  all integers $h\ge 0$.
\end{enumerate} 
\end{lem}
\begin{proof}
(i)  Fix  $C\in |C|$ and an integer $k>0$,   and consider the exact sequence:
\begin{equation}\label{eq: restrict C}
0\to H^0((k-1)C)\to H^0(kC)\overset{r_k}{\longrightarrow} H^0(\OO_C(kp)).
\end{equation}
 Let $\sigma\in H^0(C)$ be a section that cuts out $p$ on $C$. Then $r_k(\sigma^k)\ne 0$, hence the rank  $\rho_k$  of $r_k$ is $>0$ for all $k>0$. Set  $\bar\sigma:=r_1(\sigma)\in H^0(\OO_C(p))$; multiplication by $\bar\sigma$ gives an injection  $\IM r_k\hookrightarrow \IM r_{k+1}$, hence $\rho_k$ is a non decreasing function of $k\ge 1$.
Since $h^0(mC)=1+\sum_1^m\rho_k$, we have $h^0(mC)\ge m+1$ with equality holding if and only if  $\rho_k=1$ for all $1\le k \le m$. In particular this proves (i).
\medskip

(ii)  If $h^0(mC)\ge m+2$, then  the above remarks imply  $\rho_m\ge 2$ and $\rho_k\ge 2$ for all $k\ge m$, hence $h^0((m+h)C)=h^0(mC)+\sum_{j=1}^h \rho_{m+j}\ge m+2+2h$. 
\end{proof}

\begin{prop}\label{prop: h^0} 
In the above setup, let $m_0>0$ be an integer  such that   $h^0(\OO_C(m_0p))=2$ and $h^0(\OO_C((m_0-1)p))=1$ for all curves $C\in |C|$.  Then $$h^0(m_0C)=m_0+2.$$
\end{prop}
\begin{proof} 
Let $\epsilon \colon \tilde{S} \rightarrow S$ be the blow up of $S$ at $p$, and let $E$ be the exceptional curve. Let $\tilde{f} \colon \tilde{S} \rightarrow {\mathbb P}^1$ be the fibration induced by $|C|$, so that $C$  pulls back to $E+F$, $F$ being a fibre of $\tilde{f} $. Note that each fibre $F$ is isomorphic to the corresponding curve $C$ and $E$ cuts on $F$ the point mapping to $p$.

Applying $\tilde{f}_*$ to the exact sequence
\[
0
\rightarrow
{\mathcal O}_{\tilde{S}}((m-1)E)
\rightarrow
{\mathcal O}_{\tilde{S}}(mE)
\rightarrow
{\mathcal O}_{E}(mE)
\rightarrow
0
\]
we obtain, for all $m \in {\mathbb N}$, the exact sequence
\begin{multline}\label{eqn: 5 terms}
0
\rightarrow
\tilde{f}_* ({\mathcal O}_{\tilde{S}}((m-1)E))
\rightarrow
\tilde{f}_*({\mathcal O}_{\tilde{S}}(mE))
\rightarrow
{\mathcal O}_{{\mathbb P}^1}(-m)
\rightarrow\\
\rightarrow
R^1\tilde{f}_* ({\mathcal O}_{\tilde{S}}((m-1)E))
\rightarrow
R^1\tilde{f}_*({\mathcal O}_{\tilde{S}}(mE))
\rightarrow
0
\end{multline}

By Riemann-Roch  the assumptions imply that $h^1(\OO_F((m_0-1)p)$ is equal to $h^1(\OO_F(m_0p))$ and  is independent of the fiber $F$.  Therefore the sheaves $R^1\tilde{f}_* ({\mathcal O}_{\tilde{S}}((m_0-1)E))$ and $R^1\tilde{f}_*({\mathcal O}_{\tilde{S}}(m_0E))$ are locally free  of the same rank. 
Recall that a surjective morphism of locally free sheaves of the same rank is an isomorphism. Then the map $R^1\tilde{f}_* ({\mathcal O}_{\tilde{S}}((m_0-1)E))
\rightarrow
R^1\tilde{f}_*({\mathcal O}_{\tilde{S}}(m_0E))
$ is an isomorphism and we have a short exact sequence
\begin{equation*}
0
\rightarrow
\tilde{f}_* ({\mathcal O}_{\tilde{S}}((m_0-1)E))
\rightarrow
\tilde{f}_*({\mathcal O}_{\tilde{S}}(m_0E))
\rightarrow
{\mathcal O}_{{\mathbb P}^1}(-m_0)
\rightarrow
0
\end{equation*}
Note that $\tilde{f}_* ({\mathcal O}_{\tilde{S}}((m_0-1)E)) \cong {\mathcal O}_{{\mathbb P}^1}$, because it is a line bundle on ${\mathbb P}^1$ whose space of global sections has dimension $1$, and then the above short exact sequence becomes 
\begin{equation*}\label{eqn: 3 terms}
0
\rightarrow
{\mathcal O}_{{\mathbb P}^1}
\rightarrow
\tilde{f}_*({\mathcal O}_{\tilde{S}}(m_0E))
\rightarrow
{\mathcal O}_{{\mathbb P}^1}(-m_0)
\rightarrow
0
\end{equation*}
which  splits, because $h^1(\OO_{\pp^1}(m_0))=0$.  So 
\begin{equation}\label{eq: m0}
\tilde{f}_*({\mathcal O}_{\tilde{S}}(m_0E))\cong {\mathcal O}_{{\mathbb P}^1} \oplus {{\mathcal O}_{{\mathbb P}^1}}(-m_0).
\end{equation}

By \eqref{eq: m0} 
$h^0({\mathcal O}_S(m_0C))=
h^0({\mathcal O}_{\tilde{S}}( m_0(E+F)))=
h^0(\tilde{f}_*{\mathcal O}_{\tilde{S}}( m_0(E+F)))=
h^0({\mathcal O}_{\mathbb P^1} (m_0) \oplus {\mathcal O}_{\mathbb P^1})=m_0+2$.
\end{proof}

We note the following consequence of Proposition \ref{prop: h^0}:
\begin{cor} \label{cor: 3p}
If  $K_SC=5$, $h^0(\OO_C(K_S-3C))>0$ and   $h^0(\OO_C(3p))=1$ for  every curve $C\in |C|$, then $h^0(4C)=6$.
\end{cor}
\begin{proof} The curves of $|C|$ have genus 4 by the adjunction formula. If  $C\in |C|$ is general then  $K_C-4p=(K_S-3C)|_C$ is effective,  hence $h^0(\OO_C(4p))\ge 2$ by Riemann--Roch. By semicontinuity $h^0(\OO_C(4p))\ge 2$ for every $C\in |C|$. 
On the other hand $h^0(\OO_C(4p))\le h^0(\OO_C(3p))+1\le 2$ for every $C\in |C|$, so  $h^0(\OO_C(4p))=2$ for every $C$ and Proposition  \ref{prop: h^0} gives $h^0(4C)=6$.
\end{proof}
 
 The next result is crucial in the proof of points (iv) and (v) of Theorem \ref{thm: 42}. 
 \begin{prop}\label{prop: general}
Assume that $p$ is not a base point of $|K_S|$ and that for every $C\in |C|$ one has $K_S|_C=\OO_C(dp)$, where  $d=5$ or $d=7$. 
 Then $$h^0\left (\frac {d+3}{2}C\right)\ge \frac{d+7}{2}.$$
\end{prop}
The proof    requires some intermediate results and  takes up the rest of the section. 

\begin{lem}\label{lem: gaps for d=7}
In the assumptions of Proposition \ref{prop: general} fix an element $C$ of the pencil $|C|$. Then:
\begin{enumerate}
\item if $d=5$ the set of gaps of $p$ in $C$  is either $\{1,2,3,7\}$ or $\{1,2,4,7\}$;
\item if $d=7$ the set of gaps of $p$ in $C$  is either $\{1,2,3,4,9\}$ or $\{1,2,3,5,9\}$.
\end{enumerate}
\end{lem}
\begin{proof}
(i) By Lemma \ref{lemma: gaps come in pairs}, if $d=5$ then  $p$ has $4$ gaps in $C$, the minimal one is $1$ and the maximal one is $7$ and there is exactly one gap in each of the pairs $\{2,5\}$ and $\{3, 4\}$.
Since  $K_S|_C=\OO_C(dp)$  and $p$ is not a base point  of $|K_S|$ by assumption,  the point  $p$ is not a base point of $|\OO_C(5p)|$ hence $2$ is  a gap. So the possible sets of gaps are $\{1,2,3,7\}$ and $\{1,2,4,7\}$.
\medskip

(ii) By Lemma \ref{lemma: gaps come in pairs} if $d=7$,  $p$ has $5$ gaps in $C$, the minimal one is $1$ and the maximal one is $9$. 
Since the complement of the set of gaps is a semigroup and $3$ divides $9$, $3$ is a gap too. Moreover the remaining two gaps are one in each pair  $\{2, 7\}$,  $\{4, 5\}$, and arguing as in case (i) one sees  that $2$ is a gap. So the set of gaps is either $\{1,2,3,4,9\}$ or 
$\{1,2,3,5,9\}$.
\end{proof}

By the invariance of the parity of  theta characteristics we observe
\begin{lem}\label{lem: gaps for d=7 do not mix}
In the assumptions of Proposition \ref{prop: general}  the set of gaps of $p$ in $C$ does not depend on the choice of the curve  in the pencil. 
\end{lem}
\begin{proof}
Consider again the blow up  $\epsilon \colon \tilde{S} \rightarrow S$ of $S$ at $p$. Let $E$ be the exceptional divisor and let  $\tilde{f} \colon \tilde{S} \rightarrow {\mathbb P}^1$ be the fibration induced by $|C|$, so that $C$  pulls back to $E+F$, where $F$ is a fibre of $\tilde{f} $. The map $\epsilon$ induces an isomorphism of each fibre $F$ with the corresponding curve $C$ and $E$ cuts on $F$ the point mapping to $p$.

For $d=5$,  in the two cases of  Lemma \ref{lem: gaps for d=7}, the theta characteristic ${\mathcal O}_C(3p)$ has different parity, since  $h^0({\mathcal O}_C(3p))=1$  or $2$ respectively. Similarly, for $d=7$
 the theta characteristic ${\mathcal O}_C(4p)$ has different parity  for the two possible sets of gaps.
 
We cannot conclude just applying Harris Theorem  \ref{thm: parity of theta} to $\tilde f$, because the fibres of $\tilde{f}$ may a priori have irreducible components that are not reduced. We solve this issue by contracting some $(-2)$-curves of $S$ so that we have a fibration with irreducible fibers, as follows.
Every $C$ has a distinguished component $C_p$, the one containing $p$ (as a smooth point!), that is reduced and satisfies $K_SC_p=d$. All the remaining components $A$ of $C$   do not go through $p$, so they satisfy $K_SA=0$ and are  $(-2)$-curves that pull back to $(-2)$-curves of $\tilde S$ contained in the fibers of $\tilde f$. It is immediate to check that, since $S$ is minimal, all the $(-2)$-curves contracted by $\tilde f$ arise in this way. 

By contracting the $(-2)$-curves in the fibers of  $\tilde{S}$, we obtain a surface $\bar{S}$ and a commutative diagram
\[
\xymatrix
   { 
   \tilde{S}\ar[rr]^\pi \ar[dr]_{\tilde{f}}  & & \bar{S}\ar[dl]^{\bar{f}} \\
      & {\mathbb P}^1 & \\
   }
\]
The fibres $\bar{F}=\pi(F)$ of $\bar{f}$ are all irreducible and reduced, images of the distinguished components $C_p$. So we can apply Harris Theorem  \ref{thm: parity of theta} to $\bar{f}$ obtaining that the parity of $h^0({\mathcal O}_{\bar{F}}(4 \pi(E)))$ does not depend on the fibre.

We conclude the proof by showing that $h^0({\mathcal O}_{\bar{F}}(4\pi(E)))$ equals $h^0({\mathcal O}_{C}(4p))$ for the corresponding curve $C$ in the pencil. 

Fix a fibre $F$ and then its image $\bar{F}=\pi(F)$. Applying the functor $\pi_*$ to the exact sequence 
\[
0 
\rightarrow 
{\mathcal O}_{\tilde{S}}(-F)
\rightarrow 
{\mathcal O}_{\tilde{S}}
\rightarrow 
{\mathcal O}_{F}
\rightarrow 
0
\]
we obtain the exact sequence
\begin{equation}
\label{eq:pi*ES}
0 
\rightarrow 
\pi_*{\mathcal O}_{\tilde{S}}(-F)
\rightarrow 
\pi_*{\mathcal O}_{\tilde{S}}
\rightarrow 
\pi_*{\mathcal O}_{F}
\rightarrow 
R^1\pi_*{\mathcal O}_{\tilde{S}}(-F)=R^1\pi_*{\OO_{\tilde S}}\otimes \OO_{\pp^1}(-1)
\end{equation}
Note that $R^1\pi_*{\mathcal O}_{\tilde{S}}$ vanishes because the singularities of $\bar{S}$ are rational by construction and $\pi$ is the minimal resolution. Since $F=\pi^*\bar F$, applying the projection formula we rewrite \eqref{eq:pi*ES} as
\[
0 
\rightarrow 
{\mathcal O}_{\bar{S}}(-\bar{F})
\rightarrow 
{\mathcal O}_{\bar{S}}
\rightarrow 
\pi_*{\mathcal O}_{F}
\rightarrow 
0
\]
from which
\[
\pi_*{\mathcal O}_{F} \cong {\mathcal O}_{\bar{F}}
\]
Then by the projection formula $\pi_* {\mathcal O}_{F}(4E)=\pi_*\OO_F(\pi^*(4\pi(E)) \cong {\mathcal O}_{\bar{F}}(4\pi(E))$.

 The proof is completed by remarking that the map 
 $\epsilon$ induces isomorphisms ${\mathcal O}_{C}(4p) \cong {\mathcal O}_{F}(4E)$. 
\end{proof}

\begin{proof}[Proof of Proposition \ref{prop: general}] By Lemma   \ref{lem: gaps for d=7 do not mix},
  the set of gaps of $p$ in $C$ does not depend on the choice of the curve  in the pencil, so if $m_0$ is  the smallest non-gap then $h^0(m_0D)=m_0+2$  by Proposition \ref{prop: h^0}. Since $K_C=(d+1)p$  and $g(C)= \frac{d+3}{2}$, we have  $m_0\le \frac{d+3}{2}$, by Riemann-Roch. Then  Proposition \ref{prop: h^0} also gives the inequality $h^0(\frac {d+3}{2}C)\ge \frac{d+7}{2}$.
  \end{proof}
 
\section{Proof of Theorem \ref{thm: 42}, (iv) and (v)}
By Proposition \ref{prop: max-case} and Proposition \ref{prop: cases5} the possible cases for $d=7$ and $p_g=9$ and for   $d=5$  and $p_g\ge 5$  are the following:
\begin{itemize}
\item $d=7$, $p_g=9$, $K_SC=7$, $C^2=1$;
\item $d=5$, $p_g=7$, $K_SC=5$, $C^2=1$;
\item $d=5$, $p_g=6$, $K_SC=5$, $C^2=1$;
\item $d=5$, $p_g=5$, $K_SC=6$, $C^2=2$;
\item $d=5$, $p_g=5$, $K_SC=5$, $C^2=1$.
\end{itemize}
In order to complete the proof of Theorem \ref{thm: 42} we have to rule out all possibilities except  the last one.
We start with the cases where $p_g$ has the maximum value $d+2$.

\begin{prop}\label{prop: no d=5,7}
The cases $d=5$, $p_g=7$ and  $d=7$, $p_g=9$ of Proposition \ref{prop: max-case} do not occur.
\end{prop}
\begin{proof}
By Proposition \ref{prop: max-case}  if   $p_g=d+2$ then $C^2=1$,  $K_S=dC$ (hence $K_S|_C=\OO_C(dp)$)  and $|K_S|$ is base point free. So Proposition \ref{prop: general} applies, giving  $h^0(4C)\ge 6$ if $d=5$ and $h^0(5C)\ge 7$ if $d=7$. By Lemma \ref{lem: h^0} if $d=5$ we get $7=p_g=h^0(5C)\ge 8$ and if $d=7$ we get $9=p_g=h^0(7C)\ge 11$, against the assumptions.
\end{proof}

\subsection{Exclusion of $d=5$, $p_g=6$, $C^2=1$, $K_SC=5$}

The exclusion of this case is more involved than the previous ones and  requires some intermediate steps.
 
In this case the canonical image $\Sigma$ is a cone of degree 4 in $\pp^5$, so $K_S=4C+Z$ where $Z>0$ is contracted to the vertex of $\Sigma$. Observe  that $5=K_SC=4C^2+CZ=4+CZ$ implies $CZ=1$.
We have two cases according to whether the base point $p$  of $|C|$ belongs to  $Z$ or not. We start by ruling out the first case.
\begin{lem}\label{lem: p non in Z}
In case $p_g=6$ and $K_SC=5$, $C^2=1$  of Proposition \ref{prop: cases5}, the point $p$ does not belong to $Z$.
\end{lem}
\begin{proof}
Assume by contradiction that $p\in Z$. Then for general $C\in |C|$, $K_S|_C=\OO_C(5p)$. The linear series  $|K_S|_C|$ is base point free, since by assumption the canonical map sends $C$ 5-to-1 onto a line.  
So $p$ is not a base point of $|K_S|$ and Proposition \ref{prop: general} gives $h^0(4C)\geq 6$. Since $p_g=6$, we have $h^0(4C)=6=p_g$. 
Hence $Z$ is contained in the fixed part of $|K_S|$, against the assumption that the general canonical curve is smooth.
\end{proof}
\begin{lem}\label{lem: Z -3}
In case $p_g=6$ and $K_SC=5$, $C^2=1$  of Proposition \ref{prop: cases5},  the curve  $Z$ is isomorphic to $\pp^1$, $Z^2=-3$ and $Z$ intersects any curve of $|C|$ transversally at one point distinct from $p$.
\end{lem}
\begin{proof}
We claim that $Z$ is $2$-connected. Let $Z=A+B$ be a decomposition with $A,B>0$. Since $C$ is nef and $CZ=1$ we may assume $CA=1$, $CB=0$. Then we obtain a decomposition $K_S=(4C+A)+B$ and by the $2$-connectedness of canonical divisors $2\le(4C+A)B=AB$ as claimed.

Since $p\notin Z$ by Lemma \ref{lem: p non in Z} and $ZC=1$, there is no curve of $|C|$ that contains $Z$.
So $h^0(\OO_Z(C))\ge 2$.   By \cite[Prop. A.5.ii]{cfm}  we conclude that $Z$ has no component $\Gamma$ such that $\Gamma C=0$. Hence $Z$ is irreducible.  Then $CZ=1$ and $h^0(\OO_Z(C))\ge 2$ means that  $Z$ is a smooth rational curve. 

The adjunction formula gives $-2=K_SZ+Z^2=4+2Z^2$, namely $Z^2=-3$. 

\end{proof}

\begin{lem}\label{lem: C 2-connected}
The curves of $|C|$ are $2$-connected.
\end{lem}
\begin{proof}
Remark that any curve  $C\in |C|$ is 1-connected by  \cite[Lem.2.6]{m}.

  Assume by contradiction that there is a curve $C'\in |C|$ decomposing as   $C'=A+B$ with $AB= 1$. Let $B$ be minimal with respect to $B(C'-B)=1$. Then, by  \cite[Prop. A.4.ii]{cfm}, $B$ is 2-connected. 

Now we claim  that $A$ and $B$ have no common components. 

Write $A=A_0+\Delta$, $B=B_0+\Delta$ where $A_0,B_0,\Delta\geq  0$ and  $A_0$ and $B_0$ have no common components and assume by contradiction  that $\Delta\neq 0$. 
Then we have $1=AB=\Delta^2+\Delta A_0+\Delta B_0+A_0B_0= \Delta(C'-\Delta)+ A_0B_0 $.   Since $A_0B_0\geq 0$, we have $\Delta(C'-\Delta)\le 1$. So we conclude by the minimality of $B$ with respect to $B(C'-B)=1$ that   $B_0=0$ and  $\Delta=B<A$. 
So $C'=A+B= A_0+ 2B$.

 Since $C^2=1$ and $C$ is nef,  necessarily $CA=1, CB=0$.  But then $B(K_S-B)= B(C-B)+3BC+BZ=1+BZ$. So by 2-connectedness of the canonical divisors necessarily $BZ\ge1$. Since $B$ appears with multiplicity $2$ in $C'$, this contradicts $ZC=1$ (recall that $Z$ is not a component of $C$ by Lemma \ref{lem: Z -3}). 
 
 Hence $A$ and $B$ have no common components and therefore  intersect transversely at  a single point $q$. 
 
 The point $q$ is thus a disconnecting node of $C'$, and as such a base point of   $|K_S+C|=|5C+Z|$. Since   $q$,  being a double point of a curve in $|C|$, is distinct from $p$, we conclude that $q$ lies in  $Z$ but this is a contradiction since $Z$ is not a component of any curve in $|C|$ and $ZC=1$.
  
So no curve of $|C|$ can decompose as $C=A+B$ with $AB=1$.
\end{proof}

As in the proof of Lemma \ref{lem: gaps for d=7 do not mix} we let $ \bar S$ be the surface obtained by  blowing up $p\in S$ and contracting the $(-2)$-curves contained in the curves of $|C|$ (recall that  these curves do not pass through $p$). Let  $\bar f\colon \bar S\to \pp^1$ be the fibration induced by $|C|$: the strict transform $\bar E$ of the exceptional curve of the blow up and the strict transform $\bar Z$ of $Z$ are disjoint sections of $\bar f$ and the fibers of $\bar f$ are either irreducible or have two components, a component   that meets $\bar E$ and a component    that meets  $\bar Z$.

Note that, although the fibers of $\bar f$ are Gorenstein, it is possible that their components are not so. In order to deal with this situation, we recall   the notion of numerically $m$-connected Gorenstein curve (\cite[Def.~3.1]{cfhr}),   that generalizes the analogous definition for  curves on a smooth surface. A Gorenstein curve $D$ is numerically $m$-connected if for any generically Gorenstein subcurve $D'\subset D$  the inequality $\deg \omega_D|_{D'}-(2p_a(D')-2)\ge m$ holds. 

\begin{lem} \label{lem: KC ampio}
Let $\bar C$ be any fiber of $\bar f$. Then $\omega_{\bar C}$ is very ample.
\end{lem}
\begin{proof} As explained above, $\bar C$ is either irreducible or has two components $\bar A$ and $\bar B$ such that 
$\OO_{\bar A}( K_{\bar C})=\OO_{\bar A}( K_{\bar S})=\OO_{\bar A}(5p)$ and $\OO_{\bar B}( K_{\bar C})=\OO_{\bar B}( K_{\bar S})=\OO_{\bar B}(q)$, where $p$ denotes the  point  $\bar C\cap \bar E$  and $q$ denotes the  point  $\bar C\cap\bar Z$.
If we are in the latter case the curve $\bar C$ is $2$-connected by \cite[Lem.~4.2]{cfhr} and by Lemma \ref{lem: C 2-connected}, so $\deg \omega_{\bar C}|_{\bar B}-2p_a(\bar B)+2=3-2p_a(\bar B)\ge 2$, hence $p_a(\bar B)=0$.

Since $\bar A$ and $\bar B$ have no common components one has the following sequence (cf. \cite[Lem.~2.4]{cfpr}):
$$0\to \omega_{\bar A}\to \omega_{\bar C}\to \omega_{\bar C}|_{\bar B}\to 0, $$
that gives $\deg \omega_{\bar C}=2p_a(\bar A)-2+\deg \omega_{\bar C}|_{\bar B}$ and also, switching the roles of $\bar A$ and $\bar B$, 
 $\deg \omega_{\bar C}=2p_a(\bar B)-2+\deg \omega_{\bar C}|_{\bar A}$. Equating these two expressions we obtain $7-2p_a(\bar A)=\deg \omega_{\bar C}|_{\bar A}-2p_a(\bar A)+2=\deg \omega_{\bar C}|_{\bar B}-2p_a(\bar B)+2=3$,  hence $p_a(\bar A)=2$ and  $\bar C$ is numerically $3$-connected. 
 So  \cite[Thm.~3.6]{cfhr} applies and either $\omega_{\bar C}$ is very ample or $\bar C$ is honestly hyperelliptic, i.e.,  there is
a finite degree two map $\psi\colon \bar C\to\pp^1$ (an ``honest $g^1_2$''). By the above discussion, if $\bar C$ is reducible then it cannot be honestly hyperelliptic because $p_a(\bar A)\ne p_a(\bar B)$, while a reducible honestly hyperelliptic curve consists of two components both isomorphic to $\pp^1$. If $\bar C$ is irreducible then $5p+q$ is a canonical divisor. If $\bar C$ were honestly hyperelliptic, then the canonical map would be composed with $\psi$, hence $p+q$ and $2p$ would be linearly equivalent, contradicting $p\ne q$ and $p_a(\bar C)>0$.
\end{proof}

\begin{lem}\label{lem: 3p}
 Let $\bar C$ be any fiber of $\bar f$. Then $h^0(\OO_{\bar C}(3p))=1$.
\end{lem}
\begin{proof} By Lemma \ref{lem: KC ampio},  for any pair of smooth points $x,y\in \bar C$  one has $h^0(\OO_{\bar C}(x+y))=1$. So, in particular $h^0(\OO_{\bar C}(2p))=h^0(\OO_{\bar C}(p+q))=1$ and $h^0(\OO_{\bar C}(3p))=h^0(\OO_{\bar C}(2p+q))\le 2$, where the last equality is given by Serre duality, since $5p+q$ is a canonical divisor of $\bar C$.

Assume $h^0(\OO_{\bar C}(3p))=h^0(\OO_{\bar C}(2p+q))= 2$. By the previous arguments  $|\OO_{\bar C}(3p)|$ and $|\OO_{\bar C}(2p+q)|$ are base point free $g^1_3$'s and are distinct since $p\ne q$. Let $h\colon \bar C\to\pp^1\times \pp^1$ be the corresponding morphism and let $D$ be the image.  

If $\bar C$ is irreducible, then $h$ is birational and  $D$ is  of type $(3,3)$.   By adjunction $p_a(D)=4=p_a(\bar C)$, hence $h$ is an isomorphism. On the other hand, both $g^1_3$'s are ramified at $p$, hence $h(p)\in D$ is singular while $p\in \bar C$ is smooth. This contradiction shows that $h^0(\OO_{\bar C}(3p))=1$ when $\bar C$ is irreducible.

Consider now the case when $\bar C$ is reducible. We have seen in the proof of Lemma \ref{lem: KC ampio} that in this case $\bar C =\bar A\cup \bar B$ with $\bar A$, $\bar B$ irreducible with $p_a(\bar A)=2$ and $p_a(\bar B)=0$ and such that $p\in \bar A$ and $q\in \bar B$. The free linear series $|\OO_{\bar C}(2p+q)|$ and$|\OO_{\bar C}(3p)|$ restrict to $|\OO_{\bar A}(2p)|$ and $|\OO_{\bar A}(3p)|$.  Hence both  $|\OO_{\bar A}(2p)|$ and $|\OO_{\bar A}(3p)|$  are free, but this is impossible. This last contradiction completes the proof.
\end{proof}

We can now exclude this case.

\begin{prop}\label{prop: no pg=6}
The case   $p_g=6$ and $K_SC=5$, $C^2=1$  of Proposition \ref{prop: cases5} does not occur.
\end{prop}

\begin{proof}[Proof of Proposition \ref{prop: no pg=6}]
Arguing as in the proof of Lemma \ref{lem: gaps for d=7 do not mix} one shows that for all $C\in|C|$ for every $m\in \N$ we have the equality $h^0(\OO_C(mp))=h^0(\OO_{\bar C}(mp))$, where $\bar C$ is the strict transform of $C$ in $\bar S$.
So Lemma \ref{lem: 3p} gives   $h^0(\OO_{C}(3p))=1$  for all $C$ and Corollary \ref{cor: 3p} gives $h^0({\mathcal O}_S(4C))=6=p_g(S)$. So $Z$ is contained in the fixed part of $|K_S|=|4C+Z|$, contradicting the assumption that the general canonical curve is smooth.
\end{proof}

\subsection{Exclusion of $d=5$, $p_g=5$, $C^2=2$, $K_S=3C$}

\subsubsection{Preliminary considerations.}

\begin{prop}\label{prop: II-regular}
In the assumptions of Theorem  \ref{thm: 42}, if $d=5$, $p_g=5$, $C^2=2$ then $q(S)=0$.
\end{prop}
\begin{proof}
We write $q:=q(S)$ for simplicity. 

\noindent \underline{Step 1: $q\le 3$.}\\ 
For $C$ general consider the sequence:
$$0\to H^0(K_S)\to H^0(K_S+C)\to H^0(K_C)\to H^1(K_S)\to H^1(K_S+C)=0,$$
where the last term is zero by Kawamata-Viehweg vanishing, since $C$ is nef and big. 
By assumption the image of the restriction map $H^0(K_S)\to H^0(K_S|_C)$ is 2-dimensional. Since $|C|$ is a linear pencil, then  the image of 
$H^0(K_S+C)\to H^0(K_C)$ has dimension $\ge 2$, so $q=h^1(K_S)\le g(C)-2=3$.
\medskip

\noindent\underline{Step 2: $q\le 1$.}\\ Assume $q=2$ or $3$. Fix a smooth $C$. Since $C$ is nef and big the  map $\Pic^0(S) \to J(C)$ is actually an injection (see for instance \cite[Prop.~1.6]{cfm}). So  the image of the map $\phi_1\colon \Pic^0(S)\to J^2(C)$ defined by $\alpha\mapsto (C+\alpha)|_C$ is isomorphic to an abelian subvariety $A\subset J(C)$ of dimension $q$. The locus $W_2$ of effective classes in $J^2(C)$ is a surface birational
 to the symmetric product $S^2C$. So if  $q=3$ then  $W_2$ cannot contain  the image of $\phi_1$ by dimension reasons, and if $q=2$ it cannot contain it because $W_2$ is of general type while the image of $\phi_1$ is an abelian surface. 
So $h^0((C+\alpha)|_C) =0$ for $\alpha$ general. The exact sequence 
$$0\to H^0(\alpha) \to H^0(C+\alpha)\to H^0((C+\alpha)|_C)$$
gives $h^0(C+\alpha)=0$ for $\alpha\in \Pic^0(S)$ general.
 Consider now the map   $\phi_2\colon \Pic^0(S)\to J^4(C)$ defined by $\alpha\mapsto (2C+\alpha)|_C$. The image of $\phi_2$ is invariant under translation by $A$, hence if $q=3$ by
\cite[Proposition 3.3]{DF} it cannot be contained in the set of effective classes of $J^4(C)$, which is a translate of the $\Theta$-divisor of $J(C)$. 
 As a consequence, we have $h^0((2C+\alpha)|_C)=0$ for $\alpha$ general, and as above we get $h^0(2C+\alpha)=0$.
If $q=2$, \cite[Proposition 3.3]{DF} again implies that $h^0((2C+\alpha)|_C)\le 1 $, and therefore $h^0(2C+\alpha) \le 1$ for $\alpha$ general. 

Using Riemann-Roch on $S$ we compute:
$$4-q=\chi(C+\alpha)\le h^0(C+\alpha)+h^0(2C-\alpha), \quad \alpha \  \mbox{general}.$$
We have shown that the right hand side is zero if $q=3$ and is $\le 1$ if $q=2$, so we get a contradiction in either case. 
\medskip

\noindent\underline{Step 3: $q=0$.} By  Step 2  it is enough to rule out $q=1$. So assume by contradiction that $q=1$ and write $E:=\Pic^0(S)=\Alb(S)$ and let $f\colon C\to E$ be the restriction of the Albanese map. If  $\phi_1(E)$ is not contained in $W_2$, then we obtain a contradiction as in Step 2. So assume that $\phi_1$ maps $E$ to  $W_2=S^2C$. Then by \cite[Corollary 3.9]{cmlp}  there is a double cover $g\colon C\to E$ such that $\phi_1= g^*\colon E\to J^1(C)$. Both $f$ and $g$ induce the same injection of $E$ into $J(C)$, so $f$ and $g$ coincide up to a translation in $E$. 

If  $F$ is  a general fiber of the Albanese map $a\colon S\to E$, then $FC=2$, and therefore the linear pencil $|C|$ induces a $g^1_2$ on $F$. 
So $F$ is hyperelliptic and the canonical map factors through the hyperelliptic involution of $F$, contradicting the assumption that the degree of the canonical map is 5, odd. So we conclude that $q=0$.
\end{proof}

Recall that we assume that the general canonical curve is smooth and the canonical map of $X$ is 5-to-1 onto the cone $\Sigma_3\subset \pp^4$ over the twisted cubic curve of $\pp^3$. The curves of $|C|$ are mapped 5-to-1 to the rulings of the cone. By Proposition \ref{prop: cases5}, $K_S=3C$.

\begin{lem}\label{lem: basics} In the assumptions of Theorem  \ref{thm: 42}, if $d=5$, $p_g=5$, $C^2=2$ then
\begin{enumerate}
\item The general $C$ is smooth and non hyperelliptic
\item If $C=A+B$, with $A, B>0$ then $AB>0$ is even
\item $|K_S|$ and $|C|$ have  exactly one common base point $p\in S$
\item The base scheme of $|K_S|$ consists of $
p$, a point $p_1$ infinitely near to $p$ and a point $p_2$ infinitely near to $p_1$
\item $|C|$ has a second base point $q$, possibly infinitely near to $p$, in which case $q\ne p_1$
\item If $Z$ is an irreducible $(-2)$-curve of $X$, then $Z$ does not go through $p$ or $q$ and is contained in  exactly one curve of $|C|$
\item Assume that there is an irreducible curve $\Gamma$ such that $\Gamma C =1$ and $|C-2\Gamma|$ is not empty. Then $q$ is infinitely near to $p$, $\Gamma$ is unique and $2\Gamma$ is contained in the unique curve $C_0$  in $|C|$ singular at $p$. 
\end{enumerate}
\end{lem}
\begin{proof} 
\begin{enumerate}
\item By Bertini's theorem any linear system without fixed part and with selfintersection at most $3$ has a smooth element.
If the general $C$ were hyperelliptic, then the canonical map of $C$ would not separate two points in the same orbit of the hyperelliptic involution. A fortiori, neither $H^0(S,K_S+C)$ nor $H^0(S,K_S)$ would separate those points. This implies that the canonical map of $S$ restricted to $C$ has even degree, a contradiction.  
\item As $K_S=3C$, the $2-$connectedness of the canonical divisors implies $AB>0$. 
By the genus formula $A(K_S+A)=4A^2+3AB$ is even, and therefore $AB$ is even too.
\item As $C^2=2$, $|C|$ has two simple base points, possibly infinitely near.  The canonical system $|K_S|$ has 
$(3C)^2-(\deg \varphi_K) (\deg \varphi_K(S))=3$ base points, possibly infinitely near. As $K_S=3C$ every base point of $|K_S|$ on $S$ is also a base point of $|C|$. In particular $|C|$ and $|K_S|$ have at least a common base point $p$ on $S$. However, if the whole base scheme of $|C|$, of length $2$, were contained in the base scheme of $|K_S|$, then the restriction to $C$ of the canonical system of $S$ would also have base locus of length at least $2$, and then the canonical map of $S$ restricted to $C$ would have degree at most $6-2=4$ a contradiction. 
\item[(iv)-(v)] 
Then $|K_S|$ has only one base point on $S$, that we call $p$, and the remaining two base points $p_1$, $p_2$ are infinitely near to $p$. The second base point of $C$, that we call $q$, if infinitely near to $p$, cannot be $p_1$ or $p_2$ as we have just shown that the base scheme of $|C|$ is not contained in the base scheme of $|K_S|$. Consider the blow up $\beta_p \colon S_p \rightarrow S$ with exceptional divisor $E_p$. The movable part of $|K_{S_p}|$, $|\beta_p^* K_S - E_p|$, has intersection $1$ with $E_p$, so it cannot have two distinct base points on $E_p$. This shows that  $p_2$ is infinitely near to the $p_1$.
\item[(vi)] This follows immediately by $CZ=0$. 
\item[(vii)] As $\Gamma <C$, the condition $\Gamma C =1$ forces $\Gamma$ to contain either $p$ or $q$. If $q$ is not infinitely near to $P$, then every $C$ in $|C|$ is smooth at $p$ and $q$, a contradiction. Otherwise, there is exactly one curve $C_0$ that is singular at $p$, so $\Gamma \subset C_0$.
The uniqueness of $\Gamma$ is now obvious.
\end{enumerate}
\end{proof}

\subsubsection{Linear systems on the curves of $|C|$}

Let $\epsilon\colon \tilde S\to S$ be the blow-up at $p$ and $q$ and let $E_p$, $E_q$ be the corresponding $(-1)$-curves, where $E_p$ is reducible if and only if $q$ is infinitely near to $p$, in which case  $A:=E_p-E_q$ is an irreducible $(-2)$-curve. 

We denote by $f\colon \tilde S \to \pp^1$ the fibration induced by $|C|$, with fibre $F:=\epsilon^*C-E_p-E_q$. Each fibre $F$ maps to a curve $C$, and the map is an isomorphism unless $q$ is infinitely near to $p$ and $C$ is the unique element $C_0$ of $|C|$ singular at $p$. Then the fibre, say $F_0$, splits as $F_0=C'_0+A$, where $C'_0$ is the strict transform of $C_0$. In addition $p_a(C_0')=5-1=4$ and $AC'_0=2$.

For all curves $F$ in $|F|$ we will denote by $p$, respectively $q$, the line bundle restriction of $E_p$, respectively $E_q$ to $F$. For $F$ general, we can see $p$ and $q$ as simple points, the intersection of $F$ with $E_p$, respectively $E_q$.  

In this section we study the linear systems $|ap+bq|$ on the curves $F$. \par

By adjunction, we have for every curve of $|F|$:
$$K_{\tilde S}|_F=K_F=4(p+q), \quad \epsilon^*K_S|_F=3(p+q), \quad  |K_S|_F=p+|V|,$$
where $|V|\subset |2p+3q|$ is a pencil. There is exactly one curve $F_0$ for which $|V|$ is not base point free, the unique curve $F_0$ through $p_1$.
If $q$ is infinitely near to $p$, then $F_0$ is  the already mentioned curve containing $A$, mapping to the curve $C_0$ having a double point at $p$. 

\begin{lem} \label{lem: 3P+3Q}
For every  $F \in |F|$ we have:
 \begin{enumerate}
\item $ h^0(F, 3(p+q))=3$,  $h^0(F, p+q)=1$;
\item $h^0(F,2(p+q))=h^0(F, 2p+3q)-1\le 2$.
 \end{enumerate}
\end{lem}
\begin{proof}
(i) By Serre duality, $h^1(\OO_{\tilde S}(E_p+E_q))=h^1(\epsi^*K_S)$. In addition $h^1(\epsi^*K_S)=h^1(K_S)=0$, where the first equality holds because $R^1\epsilon_*\OO_{\tilde S}=0$ and the second one because of Proposition \ref{prop: II-regular}. So the sequence
$$0\to H^0(E_p+E_q)\to H^0(\epsilon^*C)\to H^0(F,p+q)\to 0$$ 
is exact for every $F\in |F|$ and $h^0(F, p+q)=1$. Then Riemann-Roch gives $ h^0(F, 3(p+q))=3$.

(ii) The curve $E_q$ meets every  fiber $F\in |F|$ at a smooth point, which is not a base point of $\epsi^*|K_S|$ by Lemma \ref{lem: basics}. A fortiori,  $q$ is not a base point of  $|2p+3q|$, namely $h^0(F,2(p+q))=h^0(F, 2p+3q)-1$. 
The inequality $h^0(F, 2p+3q)\le  h^0(F, 3(p+q))$ is obvious except for the case when $q$ is infinitely near to $p$ and $F$ is the fiber $F_0$ dominating the unique curve $C_0$  in $|C|$ singular at $p$.  In that case, let $s\in H^0(S,A)$ be a non-zero section and consider the map $\mu\colon H^0(F_0, 2p+3q)\to H^0(F_0, 3p+2q)$ defined by multiplying by $s$. Any section $\sigma\in \ker \mu$ is supported on $A$ and so vanishes on $B|_A$, where $B=F_0-A$. By Lemma \ref{lem: basics}, $B|_A$ is a scheme of length at least $2$, while $2p+3q$ has degree 1 on $A$. So $\sigma=0$, $\mu$ is injective and $h^0(F_0, 2p+3q)\le h^0(F_0, 3p+2q)$. In turn, $h^0(F_0, 3p+2q)\le h^0(F_0, 3p+3q)$ since $E_q$ meets $F_0$ at a smooth point. So $h^0(F, 2p+3q)\le  h^0(F, 3(p+q))$ for any $F\in |F|$ and $h^0(F, 2p+3q)\le 3$ by (i).
\end{proof}

We look now at $|2(p+q)|$, which is a theta characteristic on $F$.

\begin{lem} \label{lem: cases}
There are the following possibilities:
\begin{itemize}
\item[(E)] $h^0(F, 2(p+q))=2$ for every $F\in |F|$ (``even  case'');
\item[(O)] $h^0(F, 2(p+q))=1$ for every $F\in |F|$ (``odd case'');
\item[(VO)] $q$ is infinitely near to $p$, $h^0(F, 2(p+q))=1$ for every $F\ne F_0$, $h^0(F_0,2(p+q))=2$ and $C_0\ge 2\Gamma$, where $\Gamma$ is an irreducible curve such that $C\Gamma=1$ (``very odd case'').
\end{itemize}
\end{lem}
\begin{proof} 
Recall that $h^0(F,2(p+q))\le 2$ by  Lemma \ref{lem: 3P+3Q}. 
If $h^0(F,2(p+q))=2$ for $F$ general, then $h^0(F,2(p+q))=2$ for all $F$ by semicontinuity.

So assume that $h^0(F,2(p+q))=1$ for $F$ general. By Lemma \ref{lem: basics} if a fiber $\bar F\in |F|$ contains a multiple irreducible component that is not a $(-2)$-curve then $q$ is infinitely near to $p$ and $\bar F=F_0=\epsi^*C_0-E_p-E_q$.  Then we can apply Harris Theorem  \ref{thm: parity of theta} as in Lemma \ref{lem: gaps for d=7 do not mix} and obtain that the parity of $h^0(F,2(p+q))$ is constant as $F$ varies in $|F|\setminus \{F_0\}$. So, either we are in case $(O)$ or $h^0(F_0,2(p+q))=2$, and the latter case occurs only if $C_0\ge 2\Gamma$, where $\Gamma$ is an irreducible curve with $C\Gamma=1$. 
\end{proof}

\subsubsection{Exclusion of cases (E) and (O) of Lemma \ref{lem: cases}}
\begin{prop}\label{prop: noE}
Case (E) in Lemma \ref{lem: cases} does not occur.
\end{prop}

Let us first introduce the idea of the proof below. 
The movable part of the canonical system of $\tilde S$ is $|D|=|3F+2E_p+3E_q|$.
We will see that $f_*D$ is a vector bundle on ${\mathbb P}^1$ of rank $3$.
To have the canonical image equal to the cone over a twisted cubic curve, we expect  $f_*D={\mathcal O}_{{\mathbb P}^1}(3) \oplus {\mathcal O}_{{\mathbb P}^1} \oplus {\mathcal L}
$ with $h^0({\mathcal L})=0$. We will prove that such a decomposition holds but ${\mathcal L}\cong {\mathcal O}_{{\mathbb P}^1} (1)$, so $p_g=7\neq 5$. In other words, the map on the cone would not be given by the whole canonical system, but  by a proper subsystem.

\begin{proof}
We assume by contradiction that we are in case (E) of Lemma \ref{lem: basics}, so $h^0(F, 2(p+q))=2$ for all $F$. In this case $h^0(F, 2p+3q)=3$ for all $F$ by Lemma \ref{lem: 3P+3Q}.

This implies \underline{$q$ not infinitely near to $p$}. 

In fact, otherwise, choose a general $F\cong C$. Note that $p=q$ on it, we have $h^0(C,4p)=2$, $h^0(C,5p)=3$, and then, since $5$ is a prime number, $|5p|$ defines a morphism $C \rightarrow {\mathbb P}^2$ onto a quintic. Since by Riemann-Roch $h^0(C,3p)=3-5+1+h^0(C,5p)=2$, then the quintic has a singular (cuspidal) branch at (the image of) $p$. 

We set local coordinates $x,y$ at $p$ so that $p$ maps to $(0,0)$ and the line $y=0$ is the one cutting $5p$. So the local equation of the quintic is of the form $x^5=ya(x,y)$. Since the origin is a singular point that is not an ordinary double point (having a singular branch) then $a(x,y)$ is in the ideal generated by $x^2$ and $y$ and we can rewrite the quintic as 
\[
x^5=y(x^2b(x,y)+yc(x,y)).
\]
Performing a standard blow up $y\mapsto xy$ we obtain $x^5=xy(x^2b(x,xy)+xyc(x,xy))$ from which, dividing by $x^2$
\[
x^3=y(xb(x,xy)+yc(x,xy)).
\]
This equation is singular at the origin. So the quintic has at least two singular points, the image of $p$ and a point infinitely near to it. Since the arithmetic genus of a plane quintic is $6$ then its normalization has genus $\leq 4$, whereas $C$ has genus $5$, a contradiction. 

Now we show that \underline{$h^0(F, p+2q)=1$ for all $F$}. 

Since $q$ is not infinitely near to $p$, $p$ and $q$ are distinct points in all curves $C \cong F$.
Since $h^0(F, 2p+3q)=3$, we claim that $|2p+3q|$ separates $p$ and $q$. Otherwise, since $h^0(F, 2p+2q)=2$ for all $F$, there is a curve $F$ such that, on $F$, $p$ is a base point of $|2p+2q|$. On the other hand, by Riemann-Roch $h^0(F,2p+q)=3-5+1+h^0(F,2p+3q)=2$ and so $q$ is also a base point of  $|2p+2q|$. By $p\neq q$ we deduce $h^0(F,p+q)=2$ contradicting Lemma \ref{lem: 3P+3Q}.

\

Now we set $D=3F+2E_p+3E_q$. 

Recall that we have proved that $E_p$ and $E_q$ are two disjoint sections of the fibration $f\colon \tilde{S} \rightarrow {\mathbb P}^1$. Since $D\cdot E_p=1$ and $D\cdot E_q=0$, the exact sequence
\[ 0 \rightarrow D-E_p-E_q \rightarrow D \rightarrow D_{|E_p+E_q} \rightarrow 0  \]
pushes forward  to
\begin{multline}\label{eq: long exact sequence even}
0 \rightarrow f_* (3F+E_p+2E_q) \rightarrow f_* (3F+2E_p+3E_q)
 \rightarrow {\mathcal O}_{{\mathbb P}^1}(1) \oplus {\mathcal O}_{{\mathbb P}^1} 
 \rightarrow \\
 \rightarrow R^1f_* (3F+E_p+2E_q) \rightarrow R^1f_* (3F+2E_p+3E_q) \rightarrow 0
\end{multline}
The map among the $R^1$'s is an isomorphism by the same argument used in the proof of Proposition \ref{prop: h^0} about the exact sequence \eqref{eqn: 5 terms}, since, for all $F$, $h^0(F,p+2q)=1$, $h^0(F,2p+3q)=3$ and then $h^1(F,p+2q)=h^1(F,2p+3q)=2$. As there we deduce a short exact sequence of locally free sheaves 
\begin{equation}\label{eq: short exact sequence even}
0 \rightarrow f_* (3F+E_p+2E_q) \rightarrow f_* (3F+2E_p+3E_q) \rightarrow {\mathcal O}_{{\mathbb P}^1}(1) \oplus {\mathcal O}_{{\mathbb P}^1}
 \rightarrow 0
\end{equation}

Since $h^0(E_p+2E_q)=1$ the line bundle $f_* (E_p+2E_q)$ is trivial and therefore $f_* (3F+E_p+2E_q)=f_* (E_p+2E_q) \otimes {\mathcal O}_{{\mathbb P}^1}(3)={\mathcal O}_{{\mathbb P}^1}(3)$.

Then \eqref{eq: short exact sequence even} implies $h^0(3F+2E_p+3E_q)=7$. This contradicts $p_g=5$ since $|3F+2E_p+3E_q|$ is the movable part of the canonical system of $\tilde{S}$.
\end{proof}
\begin{prop}\label{prop: noO}
Case (O)  of   Lemma \ref{lem: cases} does not occur.
\end{prop}
The idea of the proof is similar to the previous one. 
\begin{proof}

We argue by contradiction. The assumption is that $h^0(F, 2(p+q))=1$ for all $F$.

\

We show that \underline{$q$ is infinitely near to $p$}.

Otherwise $E_p$ and $E_q$ are disjoint sections of $f$.
Set $D=3F+3E_p+3E_q$ and consider the exact sequence
\[ 0 \rightarrow D-E_p-E_q \rightarrow D \rightarrow D_{|E_p+E_q} \rightarrow 0  \]
 Since $D\cdot E_p=D\cdot E_q=0$, this sequence pushes forward to
\begin{multline}\label{eq: long exact sequence odd}
0 \rightarrow f_* (3F+2E_p+2E_q) \rightarrow f_* (3F+3E_p+3E_q) \rightarrow {\mathcal O}_{{\mathbb P}^1} \oplus {\mathcal O}_{{\mathbb P}^1}
 \rightarrow \\
\rightarrow R^1f_* (3F+2E_p+2E_q) \rightarrow R^1f_* (3F+3E_p+3E_q) \rightarrow 0
\end{multline}
Since all $h^q(F,2p+2q)$ and $h^q(F,3p+3q)$ do not depend on $F$, all sheaves in \eqref{eq: long exact sequence odd} are locally free. Arguing as in the proof of Proposition \ref{prop: noE}
we obtain the short exact sequence of vector bundles
\begin{equation}\label{eq: short exact sequence odd}
0 \rightarrow f_* (3F+2E_p+2E_q) \rightarrow f_* (3F+3E_p+3E_q) \rightarrow {\mathcal O}_{{\mathbb P}^1} \oplus {\mathcal O}_{{\mathbb P}^1}
 \rightarrow 0
\end{equation}
and $f_* (3F+2E_p+2E_q)$ has rank $1$. Since $h^0(\tilde{S},2E_p+2E_q)=1$ then $f_* (3F+2E_p+2E_q)=f_* (2E_p+2E_q) \otimes {\mathcal O}_{{\mathbb P}^1}(3)={\mathcal O}_{{\mathbb P}^1}(3)$.
Then \eqref{eq: short exact sequence odd} implies $h^0(3F+3E_p+3E_q)=6$ contradicting $p_g=5$. So $q$ is infinitely near to $p$.

\

Consider now the exact sequence
\[ 0 \rightarrow 3F+2E_p+2E_q \rightarrow 3F+2E_p+3E_q \rightarrow (3F+2E_p+3E_q)_{|E_q} \rightarrow 0  \]
It pushes forward to the exact sequence of locally free sheaves 
\begin{multline}\label{eq: long exact sequence odd 3}
0 \rightarrow f_* (3F+2E_p+2E_q) \rightarrow f_* (3F+2E_p+3E_q) \rightarrow {\mathcal O}_{{\mathbb P}^1} \rightarrow  \\
\rightarrow R^1f_* (3F+2E_p+2E_q) \rightarrow R^1f_* (3F+2E_p+3E_q) \rightarrow 0
\end{multline}

By Lemma \ref{lem: 3P+3Q} all $h^q(F,2p+3q)$ do not depend on $F$, so, again by the argument of the proof of Proposition \ref{prop: noE}, we get a short exact sequence of locally free sheaves
\begin{equation}\label{eq: short exact sequence odd 3}
0 \rightarrow f_* (3F+2E_p+2E_q) \rightarrow f_* (3F+2E_p+3E_q) \rightarrow {\mathcal O}_{{\mathbb P}^1}  \rightarrow 0
\end{equation}
and $f_* (3F+2E_p+2E_q) \cong {\mathcal O}_{{\mathbb P}^1}(3)$.

\

Since $Ext^1({\mathcal O}_{{\mathbb P}^1},f_* (3F+2E_p+2E_q) )=H^1({\mathcal O}_{{\mathbb P}^1}(3))=0$ we deduce from \eqref{eq: short exact sequence odd 3} the isomorphism
\begin{equation}\label{eq: 2p+3q}
f_* (3F+2E_p+3E_q) \cong {\mathcal O}_{{\mathbb P}^1}(3) \oplus {\mathcal O}_{{\mathbb P}^1}
\end{equation}

Observe that by $h^0(F,2p+2q)=1$ it follows $h^0(F,p+2q)=1$ for all $F$: this follows, for $F=F_0$, by the argument at the end of the proof of Lemma \ref{lem: 3P+3Q}, whereas it is obvious for  $F \ne F_0 $.  Then by Serre Duality and Riemann-Roch Theorem, for all $F$, $h^1(F,3p+2q)=1$ and $h^0(F,3p+2q)=2$.

The multiplication by an equation of $A=E_p-E_q$ gives an injective morphism $f_* (3F+2E_p+3E_q) \rightarrow f_* (3F+3E_p+2E_q)$ between vector bundles of rank $2$. The corresponding map on global sections is an isomorphism, since $h^0(3F+2E_p+3E_q)=5$ and $h^0(3F+3E_p+2E_q)\le p_g=5$.

By  \eqref{eq: 2p+3q} then also
\begin{equation}\label{eq: 3p+2q}
f_* (3F+3E_p+2E_q) \cong {\mathcal O}_{{\mathbb P}^1}(3) \oplus {\mathcal O}_{{\mathbb P}^1}
\end{equation}

Finally we push forward the short exact sequence
\[
0 \rightarrow 3F+3E_p+2E_q \rightarrow 3F+3E_p+3E_q \rightarrow (3F+3E_p+3E_q)_{|E_q} \rightarrow 0
\]
obtaining the exact sequence of locally free sheaves
\begin{multline*}
0 \rightarrow f_* (3F+3E_p+2E_q) \rightarrow f_* (3F+3E_p+3E_q) \rightarrow {\mathcal O}_{{\mathbb P}^1}
 \rightarrow \\
\rightarrow R^1f_* (3F+3E_p+2E_q) \rightarrow R^1f_* (3F+3E_p+3E_q) \rightarrow 0
\end{multline*}
and then, as in the proof of Proposition \ref{prop: noE},  the short exact sequence 
\[
0 \rightarrow f_* (3F+3E_p+2E_q) \rightarrow f_* (3F+3E_p+3E_q) \rightarrow {\mathcal O}_{{\mathbb P}^1}
 \rightarrow 0
\]
By \eqref{eq: 3p+2q}, $h^1(f_* (3F+3E_p+2E_q))=0$, and then
\[
p_g=h^0(3F+3E_p+3E_q)=h^0(3F+3E_p+2E_q)+h^0({\mathcal O}_{{\mathbb P}^1})=p_g+1.
\] 
\end{proof}

\subsubsection{Exclusion of case (VO) of Lemma \ref{lem: cases}}

We now know, by Lemma \ref{lem: cases} and Propositions \ref{prop: noE} and \ref{prop: noO}, that if $S$ is a surface in the case   $p_g=5$, $K_SC=6$ and $C^2=2$  of Proposition \ref{prop: cases5} , then the pencil  $|C|$ contains a curve $C_0=2\Gamma+R$, where $\Gamma$ is irreducible with $\Gamma C=1$ and all the components of $R$ are $(-2)$-curves. 

The first base point $p$ of $|C|$ is smooth for $\Gamma$.
We note that (by  $\Gamma C=1$) the second base point $q$ of $|C|$, infinitely near to $p$, does not belong to the strict transform of $\Gamma$.

We use $\Gamma$ to construct a new surface $Y$ with an irreducible  pencil $|G|$ such that $G^2=1$, $K_Y=7G$ and canonical map of degree 5.
We  write $C_0=2M+\Delta$, where $M, \Delta>0$ and $\Delta$ is reduced. Since $\Gamma \le M$ and  $2=C^2=2MC+C\Delta$, we see that $MC=1$ and $\Delta$ is a union of $(-2)$-curves. Set $L:=M+\Delta$, pick $C_1\in |C|$  general and let $\pi\colon X\to S$ be the double cover given by the relation $2L\sim_{lin} C_1+\Delta$. So the branching locus of $\pi$ is the union of $C_1$ and $\Delta$.

\begin{lem} The divisor $\Delta$ is a disjoint union of $(-2)$-curves $\Delta_1,\dots, \Delta_r$.
\end{lem}
\begin{proof} The divisor $\Delta$ is reduced and supported on $(-2)$-curves; we decompose $\Delta$ as a sum of disjoint connected divisors $\Delta_1,\dots \Delta_r$. Since $S$ is of general type, for $i=1,\dots r$ the dual graph of   $\Delta_i$   is a tree. On the other hand, if $A$ is a  component of $\Delta_i$  then   $0=AC=2AM+A\Delta=2AM+A\Delta_i$, so  for all components $A$ of $\Delta_i$ the intersection number $A\Delta_i$ is even. This is possible only if  $\Delta_i$ is  irreducible. 
\end{proof} 
Let $\eta\colon S\to \bar S$ be the contraction of $\Delta_1,\dots \Delta_r$ to singular points $s_1,\dots s_r$ of type $A_1$ and let $|\bar C|$ be the pencil of $\bar S$ induced by $|C|$. The  Stein factorization of $\eta\circ \pi$ gives rise to the following commutative diagram: 
\begin{equation}\label{diag: usual}
\xymatrix
   { 
    X \ar[rr]^{\eta_0} \ar[d]_{\pi}  & & Y\ar[d]^{\bar \pi} \\
     S \ar[rr]^\eta   &  & \bar S& \\
   }
\end{equation}
where $X$ and  $Y$ are smooth, $\pi$ is a flat double cover branched on $C_1+\Delta$,  $\bar \pi$ is a double cover branched on the image $\bar C_1$ of $C_1$ and on the points $s_1,\dots s_r \in \bar S$, and $\eta_0$ is the blow up of the preimages $y_1,\dots y_r$ of  $s_1,\dots s_r \in \bar S$. We denote  by $E_i$ the exceptional curve of $X$ corresponding to $y_i$ (so $\pi^*\Delta_i=2E_i$), by $x\in X$ the preimage of the base point $p$ of $|C|$ and by $y$ the image point $\eta_0(x)$. Note that, since every $(-2)$-curve of $S$ does not contain $p$, then $\bar \pi(y)$ is a smooth point of $\bar S$.  

\begin{lem}\label{lem: D} In the above set-up there is a linear pencil $|D|$ on $X$  with $D^2=1$  and base point $x$ such that $\pi^*C=2D$ and  there is a  pencil $|G|$ on $Y$  with $G^2=1$ and  base point $y$ such that $\eta_0^*G=D$, $\pi^*\bar C=2G$. 
\newline In addition, $K_X=7D+E_1+\dots +E_r$ and $K_Y=7G$ .
\end{lem}
\begin{proof} Let $C\in |C|$ be a   curve distinct  from $C_0$ and $C_1$. The  restriction of $\pi$ to $C$ is the  double cover of $C$  associated to the relation $2L|_C\sim_{lin} 2p$, hence $\pi^*C$ has an ordinary double point  at $x$. The normalization of $\pi^*C$ is the double cover given by the relation  $2(L|_C(-p))\sim_{lin}0$, because $L|_C=\Gamma|_C=\OO_C(p)$, and therefore it is the disjoint union of two copies of $C$.
Therefore $\pi^*C=D_1+D_2$, where $D_1$ and $D_2$ are isomorphic to $C$ and meet transversally at $x$.   Let $X'\to X$ be the blow up at $x$ and let $D'_1$ and $D'_2$ be the strict transforms of $D_1,D_2$ on $X'$. The divisors  $D'_1$ and $D'_2$ are disjoint, hence they vary in the same irreducible pencil as $C$ varies in $|C|$. Therefore $D_1$ and $D_2$  belong  to the same irreducible pencil of $X$.
As its self-intersection is positive (equal to 1) the pencil is linear, and  we denote it by $|D|$. The covering involution fixes two curves of $|D|$, namely the preimage $D_1$ of $C_1$ and the preimage $D_0$ of $C_0$. 
Since the $\eta_0$-exceptional curves are contained in  $D_0$ and do not contain the base point $x$ of $|D|$, $|D|$ induces a pencil
 $|G|$ of $Y$ with $G^2=1$ and $\bar \pi^*\bar C=2G$. 

By the Hurwitz formula $K_X=\pi^*K_S+ D+E_1+...+E_r$ and $K_Y=\bar \pi^*K_{\bar S}+G$. Then the formulae for $K_X$ and $K_Y$ follow from   $K_S=3C$, $K_{\bar S}=3\bar C$, $\pi^*(3C)=6D$ and  $ \bar\pi^*(3\bar C)=6G$. 
\end{proof}
In view of diagram \eqref{diag: usual} and of Lemma \ref{lem: D}, we have  $h^i(\OO_X(mD))=h^i(\OO_Y(mG))$ for all $i, m\ge 0$.
We are mainly interested in $Y$, but it is often easier to compute these invariants on $X$, since the double cover $\pi$ is flat and one can use the projection formulae.
\begin{lem}\label{lem: 6D}
In the above setup,  $8\le h^0(6G)\le 9$ and $y$ is a base point of $|6G|$. 
\end{lem}
\begin{proof}
Since $6D=\pi^*3C$,  $\pi_*\OO_X(6D)=\OO_S(3C) \oplus \OO_S(3C-L)$ by the projection formula, so
 $h^0(\OO_X(6D))=h^0(\OO_S(3C))+h^0(\OO_S(3C-L))=5+h^0(\OO_S(3C-L))$. The cohomology group  $ H^0(\OO_S(3C-L))$ is the kernel of the restriction map $H^0(3C)\to H^0(3C|_L)$ and so has dimension $\le 4$ since $|3C|$ has no fixed part.  On the other hand $h^0(3C-L)\ge h^0(2C)=3$, so  $8\le h^0(\OO_X(6D))\le 9$.
 
The decomposition of $\pi_*\OO_X(6D)$ induces a decomposition $H^0(\OO_X(6D))=\pi^*H^0(\OO_S(3C))\oplus \sigma \pi^*H^0(\OO_S(3C-L))$, where $\sigma \in H^0(\OO_X(\pi^*L))$ is a section with zero locus equal to the ramification locus $D_1+E_1+\dots +E_r$ of $\pi$.  So $x$ is a base point of $|6D|$ and $y$ is a base point of $|6G|$.
\end{proof}
Let $\tilde Y$ be the blow up of $Y$ at $y$, let  $E_y$ be the exceptional curve and let $\tilde G$ be the general fiber of the fibration  $\tilde g\colon \tilde Y\to \pp^1$  induced by $|G|$. 
For $i=0,1$ we  denote by $G_i$, resp. $\tilde G_i$, $D_i$, the pull back of $C_i$ to $Y$, resp. $\tilde Y$, $X$. 

We are going to study in detail the following exact sequence for $m\ge 1$:
 \begin{gather}\label{eq: R1}
 0\to \tilde g_*mE_y\to  \tilde g_*(m+1)E_y\to\OO_{\pp^1}(-m-1)\to \\
 \to R^1\tilde g_*mE_y\to R^1 \tilde g_*(m+1)E_y\to 0\nonumber
 \end{gather}

\begin{lem}\label{lem: G_0} 
In the above setup:
\begin{enumerate}
\item the set of gaps for $y$ in $G$ is $\{1,2,3,4,9\}$ for all $G\in |G|\setminus\{G_0\} $;
\item the set of gaps for $G_0$   is $\{1,3,5, 7,9\}$.
\end{enumerate}
\end{lem}
\begin{proof} 
 (i) If $G\ne G_0$, then $G$ is isomorphic to a curve $D\in |D|\setminus D_0$ and the isomorphism sends $y$ to $p$, so $h^0(G,4y)=1$ by Lemma \ref{lem: cases} and the set of gaps is as stated.  
 \medskip 
 
 (ii) Assume for contradiction that  $h^0(G_0,4y)=1$. Then  $h^0(G,my)$ is independent of $G\in |G|$ and  all the sheaves in sequence \eqref{eq: R1} are locally free for all $m\ge 1$. In particular  for $0\le m \le 4$,  $g_*mE_y \cong \OO_{\pp^1}$ as it is a line bundle with space of global sections of dimension $1$. Moreover for $m\ge 4$ the last map in \eqref{eq: R1} is a surjective morphism of line bundles, hence it is an isomorphism. So  the sequence $$0\to \tilde g_*mE_y\to  \tilde g_*(m+1)E_y\to\OO_{\pp^1}(-m-1) \to 0$$ is exact.  For $m=4$ we obtain $\tilde g_*5E_y=\OO_{\pp^1}\oplus \OO_{\pp^1}(-5)$, since $H^1(\OO_{\pp^1}(5))=0$ and, similarly,   
  $\tilde g_*6E_y=\OO_{\pp^1}\oplus \OO_{\pp^1}(-5)\oplus \OO_{\pp^1}(-6)$. Twisting by $\OO_{\pp^1}(6)$ gives $h^0(6G)=h^0((\tilde g_*6E_y)(6))=10$, contradicting Lemma \ref{lem: 6D}. So $h^0(G_0,4y)>1$.
   
Applying Mumford-Harris theorem on the parity of theta characteristic as in Lemma  \ref{lem: gaps for d=7 do not mix} we see that  $h^0(G_0,4y)=3$. Arguing as in Lemma \ref{lem: gaps for d=7} one shows that the set of gaps is as stated.
 \end{proof}
 
Finally we are ready to exclude this case, by studying the torsion part of the sheaves $R^1\tilde g_*mE_y$. So denote by $\tau_m$ the torsion subsheaf of  $R^1\tilde g_*mE_y$. 
By cohomology and base change, if  $\tilde G_t$ is the fiber of $\tilde g$ over some  $t\in \pp^1$ and $G_t$ is the image of $\tilde G_t$ in $Y$, then $H^1(\OO_{ G_t}(my))\cong R^1\tilde g_*mE_y\otimes \OO_{\pp^1}/I_t$. 
So by Lemma \ref{lem: G_0}, $\tau_m$ is supported, for all $m$, on the image point $O\in \pp^1$ of $\tilde G_0$. 
More precisely, $\tau_m\ne 0$ if and only if $2\le m\le 6$ and $\tau_m$ is a cyclic module except for $m=4$, in which case, by Theorem \ref{thm: infinitesimal},  $\tau_4 \cong \tau \oplus \tau$ with $\tau$ cyclic. 

 \begin{prop}\label{prop: no C^2=2}
The case   $p_g=5$, $K_SC=6$ and $C^2=2$  of Proposition \ref{prop: cases5} does not occur.
\end{prop}

\begin{proof}
As in the proof of Lemma \ref{lem: G_0}, (ii),  for $0\le m \le 4$,  $g_*mE_y \cong \OO_{\pp^1}$.

We consider again  sequence \eqref{eq: R1} for $m=4$.  The map $R^1\tilde g_*4E_y\to R^1 \tilde g_*5E_y$ is an isomorphism modulo torsion, since it is a surjective map of rank 1 sheaves, but it is  not  an isomorphism since $\tau_4$ and $\tau_5$ are not isomorphic, as remarked above. So \eqref{eq: R1}  splits in two short exact sequences 
\begin{equation}\label{eq: m=4 1st half}
0\to \tilde g_*4E_y\to  \tilde g_*5E_y\to\OO_{\pp^1}(-5-a)\to 0
 \end{equation}
 \begin{equation}\label{eq: m=4 2st half}
0\to {\mathbb C}[t]/t^a
\to R^1 \tilde g_*4E_y\to  R^1 \tilde g_*5E_y\to 0
 \end{equation}
  where $a>0$ is an integer. The exact sequence \eqref{eq: m=4 1st half} yields $ \tilde g_*5E_y=\OO_{\pp^1}\oplus \OO_{\pp^1}(-5-a)$ and twisting by $\OO_{\pp^1}(5)$ we obtain $h^0(5G)=6$ and also, by Lemma \ref{lem: h^0}, $h^0(mG)=m+1$ for $1\le m \le 4$. 
 Twisting   by $\OO_{\pp^1}(6)$ we obtain $h^0(6\tilde G+5E_y) = 7+h^0(\OO_{\pp^1}(1-a))\le 8$. On the other hand, $h^0(6\tilde G+5E_y)=h^0(6G)\ge 8$ by Lemma \ref{lem: 6D}, so $h^0(6G)=8$: 
 \begin{align*}
a&=1&  \tilde g_*5E_y\cong&\OO_{\pp^1}\oplus \OO_{\pp^1}(-6)
 \end{align*}
 Since both $R^1 \tilde g_*4E_y$ and $R^1 \tilde g_*5E_y$ have rank $1$, \eqref{eq: m=4 2st half} forces an exact sequence among the torsion subsheaves
  \begin{equation}\label{eq: torsion}
0\to {\mathbb C}[t]/t
\to \tau_4\to  \tau_5\to 0
 \end{equation}
 As $\tau, \tau_5$ are cyclic modules, and $\tau_4=\tau \oplus \tau$, 
 \[
 \tau \cong \tau_5 \cong  {\mathbb C}[t]/t
 \]
Since $K_{\tilde Y|\pp^1}=9\tilde G+10E_y$, by relative  duality $R^1 \tilde g_*(5E_y)$ is dual to $\tilde g_*(5E_y+9\tilde G)\cong \tilde g_*(5E_y)\otimes \OO_{\pp^1}(9)\cong \OO_{\pp^1}(9)$, hence $R^1 \tilde g_*5E_y \cong \OO_{\pp^1}(-9) \oplus \,\tau_5$. Then $h^0 (R^1 \tilde g_* (5 \tilde G + 5 E_y))=h^0(\tau_5)=1$. 
 
By the Leray spectral sequence $h^1(5 \tilde G + 5 E_y)=h^1( \tilde g_* (5 \tilde G + 5 E_y))+h^0 (R^1 \tilde g_* (5 \tilde G + 5 E_y))=1$. On the other hand $h^0(5 \tilde G + 5 E_y)=h^0(5G)=6$ and $h^2(5 \tilde G + 5 E_y)=h^0(2 \tilde G + 3 E_y)=h^0(2 \tilde G + 2 E_y)=h^0(2G)=3$. So $\chi(5 \tilde G + 5 E_y)=6-1+3=8$ which implies, by the standard Riemann-Roch Theorem for surfaces
 \[
 \chi(\OO_Y)=8+\frac12 (5 \tilde G + 5 E_y)(2 \tilde G +3 E_y)=13
 \]
 and then $h^0(7G)=p_g(Y) \ge 12$.
 
Now consider the exact sequence 
\begin{equation*}
0\to H^0(6G)\to H^0(7G)\to H^0(\OO_G(7y)),
\end{equation*}
with $G\in |G|$ a general curve.  
Since $h^0(6G)=8$ and $h^0(7G) \geq 12$,  the image of the restriction map $H^0(7G)\to H^0(\OO_G(7y))$ has dimension at least $4$ and hence the map is surjective because $h^0(\OO_G(7y))=4$. 
Therefore $y$ is not a base point of $|K_Y|$, and then Proposition \ref{prop: general} forces $h^0(5G)\ge 7$, contradicting $h^0(5G)= 6$.  
 \end{proof}

\bigskip

\bigskip

\begin{minipage}{13.0cm}

\parbox[t]{5.5cm}{Margarida Mendes Lopes\\ Centro de An\'alise Matem\'atica, Geometria e Sistemas Din\^amicos,\\
Departamento de  Matem\'atica\\
Instituto Superior T\'ecnico\\
Universidade de Lisboa\\
Av.~Rovisco Pais, 1\\
1049-001 Lisboa, Portugal\\
mmendeslopes@tecnico.ulisboa.pt
} \hfill
\parbox[t]{5.5cm}{Rita Pardini\\
Dipartimento di Matematica\\
Universit\`a di Pisa\\
Largo B. Pontecorvo, 5\\
56127 Pisa, Italy\\
rita.pardini@unipi.it}
\end{minipage}

\vskip1.0truecm

\parbox[t]{6.5cm}{Roberto Pignatelli\\
Dipartimento di Matematica\\
Universit\`a di Trento\\
via Sommarive 14\\
38123 Trento, Italy\\
roberto.pignatelli@unitn.it
 }

\end{document}